\documentclass[a4paper,12pt,intlimits,oneside,reqno]{amsart}
\usepackage{enumerate}
\usepackage{amsfonts,amsmath}

\usepackage{latexsym,amssymb}
\newtheorem{theorem}{Theorem}[section]

\theoremstyle{corollary}
\newtheorem{corollary}[theorem]{Corollary}
\theoremstyle{proposition}
\newtheorem{proposition}[theorem]{Proposition}

\theoremstyle{definition}
\newtheorem{definition}[theorem]{Definition}

\theoremstyle{remark}

\numberwithin{equation}{section}

\newcommand{\comment}[1]{}

\begin{document}

\title [p-adic rough multilinear  Hausdorff operators and commutator]{Some estimates for p-adic rough multilinear  Hausdorff operators and commutators on weighted Morrey-Herz type spaces}

\thanks{This paper is funded by Vietnam National Foundation for Science and Technology Development (NAFOSTED)}

\author{Nguyen Minh Chuong}

\address{Institute of mathematics, Vietnamese  Academy of Science and Technology,  Hanoi, Vietnam.}
\email{nmchuong@math.ac.vn}

\author{Dao Van Duong}
\address{School of Mathematics, Mientrung University of Civil Engineering, Phu Yen, Vietnam}
\email{daovanduong@muce.edu.vn}

\author{Kieu Huu Dung}
\address{School of Mathematics, University of Transport and Communications, Ha Noi, Vietnam}
\email{khdung@utc2.edu.vn}
\keywords{Rough multilinear  Hausdorff operator, commutator, central BMO space, Morrey-Herz space, $p$-adic analysis.}
\subjclass[2000]{42B25, 42B99, 26D15}
\begin{abstract}
The aim of this paper is to introduce and study the boundedness of a new class of $p$-adic rough multilinear Hausdorff operators on the product of Herz, central Morrey and Morrey-Herz spaces with power weights and  Muckenhoupt weights. We also establish the boundedness for the commutators of $p$-adic rough multilinear Hausdorff operators on the weighted spaces with symbols in central BMO space.
\end{abstract}

\maketitle
\section{Introduction}
The $p$-adic analysis in the past decades has received a lot of attention due to its important applications in mathematical physics as well as its necessity in sciences and technologies (see e.g. \cite{Avetisov1, Avetisov2, Beloshapka, Chuong1, Chuong2, Dragovich, Khrennikov1, Kozyrev, Varadarajan, Vladimirov, Vladimirov1, Vladimirov2} and references therein). All these developments have been motivated for two physical ideas. The first is the conjecture in particle physics that at very small, so-called Planck distances physical space-time has a complicated non-Archimedean structure and that $p$-adic numbers correctly reflect this structure. As a consequence of this idea have emerged the $p$-adic quantum mechanics and $p$-adic quantum field theory. The second idea comes from statistical physics, in particular in connection with models describing relaxation in glasses, macromolecules and proteins. 
\vskip 5pt
It is also known that the theory of functions from $\mathbb Q_p$ into $\mathbb C$ play an important role in $p$-adic quantum mechanics, the theory of $p$-adic probability in which real-valued random variables have to be considered to solve covariance problems. In recent years, there is an increasing interest in the study of harmonic analysis and wavelet analysis over the $p$-adic fields (see e.g. \cite{Albeverio, Chuong1, Chuong2, Chuong5, Haran1, Haran2, Khrennikov2, Kozyrev}). One of the important operators in harmonic analysis is the Hausdorff operator, and it is used to solve certain classical problems in analysis.  In 2010, Volosivets \cite{Volosivets1} introduced the Hausdorff operator on the $p$-adic field as follows
\begin{equation}\label{HausdorfVolosivets}
{\mathcal H}_{\varphi,A}(f)(x)=\int_{\mathbb Q_p^n} \varphi(t)f(A(t)x)dt,\;\;x\in\mathbb Q^n_p,
\end{equation}
where $\varphi(t)$ is a locally intefrable function on $\mathbb Q_p^n$ and $A(t)$ is an $n\times n$ invertible matrix for almost everywhere $t$ in the support of $\varphi$.  It should be pointed out that if we take  $\varphi(t)=\psi(t_1)\chi_{{\mathbb Z^*_p}^n}(t)$ and $A(t)= t_1.I_n$ ($I_n$ is an identity matrix), for $t=(t_1,t_2,...,t_n)$, where $\psi:\mathbb Q_p\to \mathbb C$ is a measurable function,  ${\mathcal H}_{\varphi,A}$ then reduces to the $p$-adic weighted Hardy-Littlewood average operator due to Rim and Lee \cite{Rim}.
For all we know, the theory of  the Hardy operators, the Hausdorff operators over the $p$-adic numbers field has been significantly developed into different contexts, and they are actually useful for $p$-adic analysis (see e.g. \cite{Chuong3, Chuongduong, Hung, Volosivets2, Volosivets3, WMF2013}). In  2012, Chen, Fan and Li \cite{CFL2012} introduced another version of Hausdorff operators, so-called the rough Hausdorff operators, as follow  
\begin{equation}\label{RoughHausdorf1}
{{\mathcal H}}_{\Phi,\Omega}(f)(x)=\int_{\mathbb R^n}\dfrac{\Phi\left(x|y|^{-1}\right)}{|y|^n}\Omega\left(y|y|^{-1}\right)f(y)dy,\;\;x\in\mathbb R^n.
\end{equation}
Note that if $\Phi$ is a radial function, then by using the change of variable in polar coordinates, the operator ${{\mathcal H}}_{\Phi,\Omega}$ is rewritten under of  the form
\begin{equation}\label{RoughHausdorf}
{{\mathcal H}}_{\Phi,\Omega}(f)(x)=\int_0^{\infty}\int_{S^{n-1}} \dfrac{\Phi(t)}{t}\Omega(y')f(t^{-1}|x|y')d\sigma(y')dt.
\end{equation}
Very recently, Volosivets \cite{ Volosivets4} introduced and investigated the rough Hausdorff operators on the Lorentz space on $p$-adic fields which are defined by
\begin{equation}\label{p-adicRoughHausdorf}
{{\mathcal H}^p}_{\Phi,\Omega}(f)(x)=\int_{\mathbb Q^n_p} \dfrac{\Phi(x|y|_p)}{|y|^n_p}\Omega(y|y|_p)f(y)dy.
\end{equation}
\vskip 5pt
Motivated by above results, we shall introduce and investigate in this paper a new class of rough multilinear Hausdorff operators defined as follows.
\begin{definition}
Let $\Phi:\mathbb Q_p\longrightarrow [0, \infty) $ and   $\Omega:S_0\longrightarrow\mathbb C $ be measurable functions such that $\Omega(y)\neq 0$ for almost everywhere $y$ in $S_0$. Let $f_1, f_2, ..., f_m$ be measurable complex-valued functions on $\mathbb Q^n_p$. The $p$-adic rough multilinear Hausdorff operator is defined by

\begin{equation}\label{RHpadic}
{\mathcal H}^p_{\Phi,\Omega}(\vec f)(x)=\sum_{\gamma\in\mathbb Z}\int_{S_0} \dfrac{\Phi(p^{\gamma})}{p^{\gamma}}\Omega(y)\prod\limits_{i=1}^{m} f_i(p^{\gamma}|x|_p^{-1}y)dy,\;\;x\in\mathbb Q^n_p,
\end{equation}
for $\vec f=\big(f_1,..., f_m\big)$.
\end{definition}
 It is clear that if $m =1$, $\Omega \equiv 1$ and $\Phi(t)= |t|_p^{n-1}\chi_{B(0,1)}$, then ${\mathcal H}^p_{\Phi,\Omega}$ is  reduced to the $p$-adic Hardy operator $\mathcal H^p$ (see, for example, \cite{FWL2013}, \cite{WMF2013})
defined by
\begin{align}
\mathcal H^p(f)(x)=\dfrac{1}{|x|_p^n}\int_{B(0,|x|_p)}f(t)dt, \,x\in\mathbb Q^n_p\setminus\lbrace 0\rbrace,
\end{align}
where $B(0,|x|_p)$ is a ball in $\mathbb Q^n_p$ with center at $0$ and radius $|x|_p$.
 Let $b$ be a measurable function. We denote by $\mathcal{M}_b$ the multiplication operator defined  by $\mathcal{M}_bf (x)=b(x) f (x)$ for any measurable function $f$. If $\mathcal{H}$ is a linear operator on some measurable function space, the commutator of Coifman-Rochberg-Weiss type formed by $\mathcal{M}_b$  and $\mathcal{H}$ is defined by $[\mathcal{M}_b, \mathcal{H}]f (x)=(\mathcal{M}_b\mathcal{H}-\mathcal{H}\mathcal{M}_b) f (x)$. Next, let us give the definition for the commutators of Coifman-Rochberg-Weiss type of $p$-adic rough multilinear Hausdorff operator.
\begin{definition}
Let $\Phi, \Omega$ be as above. The Coifman-Rochberg-Weiss type commutator of $p$-adic rough multilinear Hausdorff operator is defined by

\begin{equation}\label{commuatatorRHpadic}
{\mathcal H}^p_{\Phi,\Omega,\vec b}(\vec f)(x)=\sum_{\gamma\in\mathbb Z}\int_{S_0} \dfrac{\Phi(p^{\gamma})}{p^{\gamma}}\Omega(y)\prod\limits_{i=1}^{m}\Big(b_i(x)- b_i(p^{\gamma}|x|_p^{-1}y)\Big)\prod\limits_{i=1}^{m} f_i(p^{\gamma}|x|_p^{-1}y)dy,
\end{equation}
where $x\in\mathbb Q^n_p $, $\vec b=\big(b_1,..., b_m\big)$ and $b_i$ are locally integrable functions on $\mathbb Q_p^n$ for all $i=1,...,m$.
\end{definition}

The main purpose of this paper is to extend and study the new $p$-adic rough multilinear Hausdorff operators on the $p$-adic numbers field. We then obtain the necessary and sufficient conditions for  the boundedness of ${\mathcal {H}}^p_{\Phi,\Omega}(\vec f)$ on the  product of  Herz, central Morrey and Morrey-Herz spaces on $p$-adic field. In each case, we estimate the corresponding operator norms. Moreover, the boundedness of ${\mathcal {H}}^p_{\Phi,\Omega,\vec b}(\vec f)$ on the such spaces with symbols in central BMO space is also established. It should be pointed out that all our results  are new even in the case of linear $p$-adic rough Hausdorff operators.
\vskip 5pt
Our paper is organized as follows. In Section \ref{section2}, we present some notations and preliminaries about $p$-adic analysis as well as give some definitions of the Herz, central Morrey and Morrey-Herz spaces associated with power weights and  Muckenhoupt weights. Our main theorems are given and proved in Section \ref{section3} and Section \ref{section4}.

\section{Some notations and definitions}\label{section2}
For a prime number $p$, let $\mathbb Q_p$ be the field of $p$-adic numbers. This field is the completion of the field of rational numbers $\mathbb Q$ with respect to the non-Archimedean $p$-adic norm $|\cdot|_p$. This norm is defined as follows: if $x=0$, $|0|_p=0$; if $x\not=0$ is an arbitrary rational number with the unique representation $x=p^\alpha\frac{m}{n}$, where $m, n$ are not divisible by $p$, $\alpha=\alpha(x)\in\Bbb Z$, then $|x|_p=p^{-\alpha}$. This norm satisfies the following properties:

\rm{(i)} $|x|_p\geq 0,\;\; \forall x\in\mathbb Q_p$, $|x|_p=0\Leftrightarrow x=0;$

\rm{(ii)} $|xy|_p=|x|_p|y|_p,\;\; \forall x, y\in\Bbb Q_p;$

\rm{(iii)} $|x+y|_p\leq \max(|x|_p,|y|_p),\; \forall x, y\in\mathbb Q_p$, and
when $|x|_p\not=|y|_p$, we have $|x+y|_p=\max(|x|_p,|y|_p).$\\
It is also well-known that any non-zero $p$-adic number $x\in\mathbb Q_p$ can be uniquely represented in the canonical series
\begin {equation}\label{eq1.1}
x=p^\alpha (x_0+x_1p+x_2p^2+\cdot\cdot\cdot),
\end{equation}
where $\alpha=\alpha(x)\in\Bbb Z$,\; $x_k=0,1,...,p-1,\; x_0\not= 0,\; k=0,1,...$.
This series converges in the $p$-adic norm since $|x_kp^k|_p\leq p^{-k}$.

The space $\mathbb Q^n_p=\mathbb Q_p\times\cdot\cdot\cdot\times\mathbb Q_p$ consists of all points $x=(x_1,...,x_n)$, where $x_i\in\mathbb Q_p,\; i=1,...,n,\;\; n\geq1$. The $p$-adic norm of $\mathbb Q^n_p$ is defined by
\begin {equation}\label{eq1.2}
|x|_p=\max_{1\leq j\leq n}|x_j|_p.
\end{equation}
Let 
\[B_\alpha(a)=\left\{x\in\mathbb Q^n_p : |x-a|_p\leq p^\alpha\right\}
\]
be a ball of radius $p^\alpha$ with center at $a\in\mathbb Q^n_p$. Similarly, denote by 
\[S_\alpha(a)=\left\{x\in\mathbb Q^n_p : |x-a|_p=p^\alpha\right\}\]
the sphere with center at $a\in\mathbb Q^n_p$ and radius $p^\alpha$. If $B_\alpha=B_\alpha(0), S_\alpha=S_\alpha(0)$, then for any $x_0\in\mathbb Q^n_p$ we have $x_0+B_\alpha=B_\alpha(x_0)$ and $x_0+S_\alpha=S_\alpha(x_0)$. Especially, we denote $\mathbb Z_p$ instead of $B_0$, $\mathbb Z^*_p=B_0\setminus\{0\}$ in $\mathbb Q_p$, $\mathbb Q^*_p=\mathbb Q_p\setminus\{0\}$ and $\chi_k$ be the characteristic function of the sphere $S_k$.

Since $\mathbb Q^n_p$ is a locally compact commutative group under addition, it follows from the standard theory that there exists a Haar measure $dx$ on $\mathbb Q^n_p$, which is unique up to positive constant multiple and is translation invariant. This measure is unique by normalizing $dx$ such that
\[
\int\limits\limits_{B_0}dx=|B_0|=1,
\]
where $|B|$ denotes the Haar measure of a measurable subset $B$ of $\mathbb Q^n_p$. By simple calculation, it is easy to obtain that $|B_\alpha(a)|=p^{n\alpha}, |S_\alpha(a)|=p^{n\alpha}(1-p^{-n})\simeq p^{n\alpha}$, for any $a\in\mathbb Q^n_p$.  For  $f\in L^1_{\text{loc}}(\mathbb Q^n_p)$, we have
$$\int_{\mathbb Q^n_p}f(x)dx=\lim_{\alpha\rightarrow +\infty}\int_{B_\alpha} f(x)dx=\lim_{\alpha\rightarrow +\infty}\sum_{-\infty<\gamma\leq\alpha}\int_{S_\gamma} f(x)dx.$$
In particular, if $f\in L^1(\mathbb Q^n_p)$, we can write 
$$\int_{\mathbb Q^n_p}f(x)dx=\sum_{\alpha=-\infty}^{+\infty}\int_{S_\alpha} f(x)dx,$$
and 
$$\int_{\mathbb Q^n_p}f(tx)dx=\frac{1}{|t|_p^n}\int_{\mathbb Q^n_p} f(x)dx,$$
where $t\in\mathbb Q_p\setminus\{0\}$. For a more complete introduction to the $p$-adic analysis, we refer the readers to \cite{Khrennikov1, Vladimirov2} and the references therein.

Let $\omega(x)$ be a non-negative measurable function in $\mathbb Q^n_p$. The weighted Lebesgue space $L^q_\omega(\mathbb Q^n_p) \left(0 < q < \infty\right)$ is defined to be the space of all measurable functions $f$ on $\mathbb Q^n_p$ such that
\[
\|f\|_{L^q_\omega(\mathbb Q^n_p)}=\Big(\int_{\mathbb Q^n_p}|f(x)|^q\omega(x)dx\Big)^{1/q}<\infty.
\]
The space $L^q_{\omega, \,\rm loc}(\mathbb Q^n_p)$ is defined as the set of all measurable functions $f$ on $\mathbb Q^n_ p$ satisfying $\int_{K}|f(x)|^q\omega(x)dx<\infty$ for any compact subset $K$ of $\mathbb Q^n_p$. The space $L^q_{\omega,\rm loc}(\mathbb Q^n_p\setminus\{0\})$ is also defined in a similar way as the space $L^q_{\omega,\rm loc}(\mathbb Q^n_p)$.
\vskip 5pt
Throught the whole paper, we denote by $C$ a positive geometric constant that is independent of the main parameters, but can change from line to line. Denote $\omega^{\lambda}(B)=\left(\int_B{\omega(x)}dx\right)^{\lambda}$, for $\lambda\in\mathbb R$ and $\omega(B):=\omega^1(B)$.  We also write $a \lesssim b$ to mean that there is a positive constant $C$, independent of the main parameters, such that $a\leq Cb$. The symbol $f\simeq g$ means that $f$ is equivalent to $g$ (i.e.~$C^{-1}f\leq g\leq Cf$). For any real number $\ell>1$, denote by $\ell'$ conjugate real number of $\ell$, i.e. $\frac{1}{\ell}+\frac{1}{\ell'}=1$.
\vskip 5pt
Next, let us give the definition of weighted $\lambda$-central Morrey spaces on $p$-adic field. 
\begin{definition} Let $\lambda\in\mathbb R$ and $ 1<q<\infty$.
The  weighted $\lambda$-central Morrey $p$-adic spaces ${\mathop{B}\limits^.}^{q,\lambda}_{\omega}(\mathbb Q^n_p)$ consists of all Haar measurable functions $f\in L^q_{\omega,\rm loc}(\mathbb Q_p^n)$  satisfying $\|f\|_{{\mathop{B}\limits^.}^{q,\lambda}_{\omega}(\mathbb Q^n_p)}<\infty$,
where
\begin{equation}
\|f\|_{{\mathop{B}\limits^.}^{q,\lambda}_{\omega}(\mathbb Q^n_p)}=\mathop{\rm sup}\limits_{\gamma\in \mathbb Z}\Big(\dfrac{1}{\omega(B_\gamma)^{1+\lambda q}}\int_{B_\gamma}|f(x)|^q\omega(x)dx\Big)^{1/q}.
\end{equation}
Remark that ${\mathop{B}\limits^.}^{q,\lambda}_{\omega}(\mathbb Q^n_p)$ is a Banach space and reduces to $\{0\}$ when $\lambda<-\frac{1}{q}$.
\end{definition}
We also present some definitions of the weighted  Herz $p$-adic space and Morrey-Herz $p$-adic space.
\begin{definition} Let $\beta\in\mathbb R, 0<q<\infty$ and $0<\ell <\infty$. The weighted Herz $p$-adic space  $K_{q,\omega}^{\beta,\ell}(\mathbb Q^n_p)$ is defined as the set of all Haar measurable functions $f\in L^q_{\omega,\rm loc}(\mathbb Q_p^n\setminus \{0\})$  such that $\|f\|_{K_{q,\omega}^{\beta,\ell}(\mathbb Q^n_p)}<\infty$,
where
\begin{equation}
\|f\|_{K_{q,\omega}^{\beta,\ell}(\mathbb Q^n_p)}=\Big( \sum\limits_{k=-\infty}^{\infty} p^{k\beta\ell}\|f\chi_k\|_{L^q_\omega(\mathbb Q^n_p)}^{\ell}\Big)^{1/\ell}.
\end{equation}
\end{definition}

\begin{definition} Let $\beta\in\mathbb R, 0<q<\infty$ and $0<\ell <\infty$. The weighted Herz $p$-adic space  ${\mathop{K}\limits^{.}}_{q,\omega}^{\beta,\ell}(\mathbb Q^n_p)$ is defined as the set of all Haar measurable functions $f\in L^q_{\omega,\rm loc}(\mathbb Q_p^n\setminus \{0\})$  such that $\|f\|_{{\mathop{K}\limits^{.}}_{q,\omega}^{\beta,\ell}(\mathbb Q^n_p)}<\infty$,
where
\begin{equation}
\|f\|_{{\mathop{K}\limits^{.}}_{q,\omega}^{\beta,\ell}(\mathbb Q^n_p)}=\Big( \sum\limits_{k=-\infty}^{\infty} \omega^{\beta\ell/n}(B_k)\|f\chi_k\|_{L^q_\omega(\mathbb Q^n_p)}^{\ell}\Big)^{1/\ell}.
\end{equation}
\end{definition}
\begin{definition} Let $\beta\in\mathbb R, 0<q<\infty$, $0<\ell <\infty$ and $\lambda$ be a non-negative real number. The weighted Morrey-Herz $p$-adic space is defined by
\[
MK_{\ell, q,\omega}^{\beta,\lambda}(\mathbb Q^n_p)=\left\{f\in L^q_{\omega,\rm loc}(\mathbb Q_p^n\setminus \{0\}):\|f\|_{MK_{\ell, q,\omega}^{\beta,\lambda}(\mathbb Q^n_p)}<\infty\right\},
\] 
where
\begin{equation}
\|f\|_{MK_{\ell, q,\omega}^{\beta,\lambda}(\mathbb Q^n_p)}=\mathop{\rm sup}\limits_{k_0\in\mathbb Z}p^{-k_0\lambda}\Big( \sum\limits_{k=-\infty}^{k_0} p^{k\beta\ell}\|f\chi_k\|_{L^q_\omega(\mathbb Q^n_p)}^{\ell}\Big)^{1/\ell}.
\end{equation}
\end{definition}
Let us recall to define the central BMO $p$-adic space.
\begin{definition}
Let $1\leq q<\infty$ and $\omega$ be a weight function. The central bounded mean oscillation space ${\mathop {CMO}\limits^{.}}^q_\omega(\mathbb Q^n_p)$ is defined as the set of all functions $f\in L^q_{\omega,\rm loc}(\mathbb Q^n_p)$ such that
\begin{equation}
\big\|f\big\|_{{\mathop {CMO}\limits^{.}}^q_\omega(\mathbb Q^n_p)}=\mathop {\rm sup}\limits_{\gamma\in \mathbb Z}\Big( \frac{1}{\omega(B_\gamma)}\int\limits_{B_\gamma}{|f(x)-f_{\omega, B_\gamma}|^q\omega(x)dx}\Big)^{\frac{1}{q}},
\end{equation}
where 
\[
f_{\omega, B_\gamma}=\frac{1}{\omega(B_\gamma)}\int\limits_{B_\gamma}{f(x)\omega(x)dx}.
\]
\end{definition} 
\vskip 5pt
The theory of $A_p$ weight was first introduced by Benjamin Muckenhoupt on the Euclidean spaces
in order to characterise the boundedness of Hardy-Littlewood maximal functions on the weighted $L^p$ spaces (see \cite{Muckenhoupt1972}). For $A_p$ weights on the $p$-adic fields,  more generally, on the local fields or homogeneous type spaces,  one can refer to \cite{chuonghung, HCE2012} for more details.
\begin{definition} Let $1 < \ell < \infty$. We say that a weight $\omega \in A_\ell(\mathbb Q^n_p)$ if there exists a constant $C$ such that for all balls $B$,
$$\Big(\dfrac{1}{|B|}\int_{B}\omega(x)dx\Big)\Big(\dfrac{1}{|B|}\int_{B}\omega(x)^{-1/(\ell-1)}dx\Big)^{\ell-1}\leq C.$$
We say that a weight $\omega\in A_1(\mathbb Q^n_p)$ if there is a constant $C$ such that for all balls $B$,
$$\dfrac{1}{|B|}\int_{B}\omega(x)dx\leq C\mathop{\rm essinf}\limits_{x\in B}\omega(x).$$
We denote by $A_{\infty}(\mathbb Q^n_p) = \bigcup\limits_{1\leq \ell<\infty}A_\ell(\mathbb Q^n_p)$.
\end{definition}
We give the following standard result related to the Muckenhoupt weights.
\begin{proposition}
\begin{enumerate}
\item[\rm (i)] $A_\ell(\mathbb Q^n_p)\subsetneq A_q(\mathbb Q^n_p)$, for $1\leq  \ell < q < \infty$.
\item[\rm (ii)] If $\omega\in A_\ell(\mathbb Q^n_p)$ for  $1 < \ell < \infty$, then there is an $\varepsilon > 0$ such that $\ell-\varepsilon > 1$ and $\omega\in A_{\ell-\varepsilon}(\mathbb Q^n_p)$.
\end{enumerate}
\end{proposition}

A closing relation to $A_{\infty}(\mathbb Q^n_p)$ is the reverse H\"{o}lder condition. If there exist $r > 1$ and a fixed constant $C$ such that
$$\Big(\dfrac{1}{|B|}\int_{B}\omega(x)^rdx\Big)^{1/r}\leq \dfrac{C}{|B|}\int_{B}\omega(x)dx,$$
for all balls $B \subset\mathbb Q^n_p$, we then say that $\omega$ satisfies the reverse H\"{o}lder condition of order $r$ and write $\omega\in RH_r(\mathbb Q^n_p)$. According to Theorem 19 and Corollary 21 in \cite{IMS2015}, $\omega\in A_{\infty}(\mathbb Q^n_p)$ if and
only if there exists some $r > 1$ such that $\omega\in RH_r(\mathbb Q^n_p)$. Moreover, if $\omega\in RH_r(\mathbb Q^n_p)$, $r > 1$, then $\omega\in RH_{r+\varepsilon}(\mathbb Q^n_p)$ for some $\varepsilon > 0$. We thus write $r_\omega = {\rm sup}\{r > 1: \omega\in RH_r(\mathbb Q^n_p)\}$ to denote the critical index of $\omega$ for the reverse H\"{o}lder condition.

An important example of $A_p(\mathbb Q^n_p)$ weight is the power function $|x|_p^{\alpha}$. By the similar arguments as Propositions 1.4.3 and 1.4.4 in \cite{LDY2007}, we obtain the following properties of power weights.
\begin{proposition} Let $x\in\mathbb Q^n_p$. Then, we have
\begin{itemize}
\item[(i)] $|x|^{\alpha}_p\in A_1(\mathbb Q^n_p)$ if and only if $-n< \alpha\leq 0$;
\item[(ii)] $|x|^{\alpha}_p\in A_\ell(\mathbb Q^n_p)$ for $1 < \ell< \infty$, if and only if $-n < \alpha < n(\ell-1)$.
\end{itemize}
\end{proposition}
Let us give the following standard characterization of $A_\ell$ weights which it is proved in the similar way as the real setting (see \cite{Grafakos, Stein} for more details).
\begin{proposition}\label{pro2.3DFan}
Let $\omega\in A_\ell(\mathbb Q^n_p) \cap RH_r(\mathbb Q^n_p)$, $\ell\geq 1$ and $r > 1$. Then,  there exist constants $C_1, C_2 > 0$ such that
$$
C_1\left(\dfrac{|E|}{|B|}\right)^\ell\leq \dfrac{\omega(E)}{\omega(B)}\leq C_2\left(\dfrac{|E|}{|B|}\right)^{(r-1)/r}
$$
for any measurable subset $E$ of a ball $B$.
\end{proposition}
\begin{proposition}\label{pro2.4DFan}
If $\omega\in A_\ell(\mathbb Q^n_p)$, $1 \leq \ell < \infty$, then for any $f\in L^1_{\rm loc}(\mathbb Q^n_p)$ and any ball $B \subset \mathbb Q^n_p$,
$$
\dfrac{1}{|B|}\int_{B}|f(x)|dx\leq C\left(\dfrac{1}{\omega(B)}\int_{B}|f(x)|^\ell\omega(x)dx\right)^{1/\ell}.
$$
\end{proposition}
\section{The main results about the boundness of ${\mathcal{H}}^p_{\Phi,\Omega}$}\label{section3}
Let us now assume that $q$ and $q_i\in [1,\infty)$, $\alpha,\alpha_i$, $\beta_i,\beta, \ell_i,\ell$ are real numbers such that $\alpha_i \in (-n,\infty)$, $\ell,\ell_i \in [1,\infty)$, $i=1,2,...,m$ and 
$$   \frac{1}{{{q_1}}} + \frac{1}{{{q_2}}} + \cdots + \frac{1}{{{q_m}}} = \frac{1}{q}, $$
  $$   \frac{{\alpha_1}}{{{q_{1}}}} + \frac{\alpha _2}{{{q_{2}}}} + \cdots + \frac{\alpha_m}{{{q_{m}}}} = \frac{\alpha}{q},$$
$$ \beta_1+\cdots+\beta_m = \beta, $$
$$ \dfrac{1}{\ell_1}+\cdots+\dfrac{1}{\ell_m}=\dfrac{1}{\ell}.$$
\begin{theorem}\label{TheoremMorrey}
Let $\omega_i(x)=|x|^{\alpha_i}_p$, $\lambda_i\in \big(\frac{-1}{q_{i}},0\big)$, for all $i=1, ..., m$, $\omega(x)= |x|^\alpha_p$, $\Omega\in L^{q'}(S_0)$ and the following condition is true:
\begin{align}\label{lambdaMorrey}
\frac{{n + {\alpha_1}}}{{n + \alpha}}{\lambda _1} + \frac{{n + {\alpha_2}}}{{n + \alpha }}{\lambda _2} +\cdots+ \frac{{n + {\alpha_m}}}{{n + \alpha}}{\lambda _m} = \lambda.
\end{align}
Then, we have that ${\mathcal H}^p_{\Phi,\Omega}$ is bounded from ${\mathop B\limits^.}^{q_1,\lambda _1}_{\omega_1}(\mathbb Q^n_p)\times\cdots \times {\mathop B\limits^.}^{{q_m},{\lambda _m}}_{\omega_m}(\mathbb Q^n_p)$ to ${\mathop B\limits^.}^{q,\lambda}_{\omega}(\mathbb Q^n_p)$ if and only if 
\[
\mathcal C_{1}= \sum\limits_{\gamma\in\mathbb Z} {\dfrac{\Phi(p^\gamma)}{p^{\gamma(1+(n+\alpha)\lambda)}}} <  + \infty.
\]
Furthermore, we obtain ${\big\|{\mathcal H}^p_{\Phi,\Omega}\big\|_{{\mathop B\limits^.}^{{q_1},{\lambda_1}}_{{\omega_1}}(\mathbb Q^n_p) \times  \cdots \times {\mathop B\limits^.}^{{q_m},{\lambda _m}}_{{\omega_m}}(\mathbb Q^n_p) \to {\mathop B\limits^.}^{q,\lambda }_\omega(\mathbb Q^n_p)}}\simeq \mathcal C_{1}.$
\end{theorem}
\begin{proof}
We prove the sufficient condition of the theorem.
For $R\in\mathbb Z$, denote
$$
\Delta_R =\Big(\dfrac{1}{\omega(B_R)^{1+\lambda q}}\int_{B_R}|{\mathcal H}^p_{\Phi,\Omega}(\vec f)(x)|^{q}\omega(x)dx\Big)^{1/q}.
$$
By using the Minkowski inequality, we get
\begin{align}
\Delta_R &\leq \dfrac{1}{\omega(B_R)^{1/q+\lambda}}\sum\limits_{\gamma\in\mathbb Z}\dfrac{\Phi(p^{\gamma})}{p^{\gamma}}\Big(\int_{B_R}\Big|\int_{S_0}\Omega(y)\prod\limits_{i=1}^{m} f_i(p^{\gamma}|x|_p^{-1}y)\omega^{1/q}(x)dy\Big|^qdx\Big)^{1/q}\nonumber
\\
&\leq \dfrac{1}{\omega(B_R)^{1/q+\lambda}}\sum\limits_{\gamma\in\mathbb Z}\dfrac{\Phi(p^{\gamma})}{p^{\gamma}}\Big(\int_{S_0}|\Omega(y)|\Big(\int_{B_R}\prod\limits_{i=1}^{m} |f_i(p^{\gamma}|x|_p^{-1}y)|^q\omega(x)dx\Big)^{1/q}dy\Big).\nonumber
\end{align}
Next, by applying the H\"{o}lder inequality, we have
\begin{align}\label{Holderqi}
&\int_{S_0}|\Omega(y)|\Big(\int_{B_R}\prod\limits_{i=1}^{m} |f_i(p^{\gamma}|x|_p^{-1}y)|^q\omega(x)dx\Big)^{1/q}dy \nonumber
\\
&\,\,\,\,\,\,\,\,\,\,\,\,\,\,\,\,\,\,\,\,\,\,\,\,\,\,\,\,\,\,\leq \|\Omega\|_{L^{q'}(S_0)}\Big(\int_{S_0}\int_{B_R}\prod\limits_{i=1}^{m} |f_i(p^{\gamma}|x|_p^{-1}y)|^q\omega(x)dxdy\Big)^{1/q}\nonumber
\\
&\,\,\,\,\,\,\,\,\,\,\,\,\,\,\,\,\,\,\,\,\,\,\,\,\,\,\,\,\,\,\leq \|\Omega\|_{L^{q'}(S_0)}\prod\limits_{i=1}^{m}\Big(\int_{S_0}\int_{B_R}|f_i(p^{\gamma}|x|_p^{-1}y)|^{q_i}\omega_i(x)dxdy\Big)^{1/q_i}.
\end{align}
Thus, we deduce
\begin{align}\label{Delr1}
\Delta_R\leq \dfrac{1}{\omega(B_R)^{1/q+\lambda}}\|\Omega\|_{L^{q'}(S_0)}\sum\limits_{\gamma\in\mathbb Z}\dfrac{\Phi(p^{\gamma})}{p^{\gamma}}\prod\limits_{i=1}^{m}\Big(\int_{S_0}\int_{B_R}|f_i(p^{\gamma}|x|_p^{-1}y)|^{q_i}\omega_i(x)dxdy\Big)^{1/q_i}.
\end{align}
On the other hand, $\int_{S_0}\int_{B_R}|f_i(p^{\gamma}|x|_p^{-1}y)|^{q_i}\omega_i(x)dxdy$ is equaled by
\begin{align}\label{S0Br}
&\int_{S_0}\sum\limits_{\theta\leq R}\int_{S_\theta}|f_i(p^{\gamma-\theta}y)|^{q_i}p^{\theta\alpha_i}dxdy=\int_{S_0}\sum\limits_{\theta\leq R}
|f_i(p^{\gamma-\theta}y)|^{q_i}p^{\theta\alpha_i}p^{n\theta}(1-p^{-n})dy\nonumber
\\
&= (1-p^{-n})p^{\gamma(n+\alpha_i)}\sum\limits_{\theta_1\leq R-\gamma }\int_{S_0}
|f_i(p^{-\theta_1}y)|^{q_i}p^{\theta_1\alpha_i}p^{n\theta_1}dy\nonumber
\\
&=(1-p^{-n}) p^{\gamma(n+\alpha_i)}\int_{B_{R-\gamma }}|f_i(z)|^{q_i}\omega_i(z)dz=(1-p^{-n})p^{\gamma(n+\alpha_i)}\|f_i\|^{q_i}_{L^{q_i}_{\omega_i}(B_{R-\gamma })}.
\end{align}
Hence, by (\ref{Delr1}), we have
\begin{align}\label{Delr2}
\Delta_R \leq \dfrac{1}{\omega(B_R)^{1/q+\lambda}}(1-p^{-n})\|\Omega\|_{L^{q'}(S_0)}\sum\limits_{\gamma\in\mathbb Z}\dfrac{\Phi(p^{\gamma})}{p^{\gamma(1-\frac{n}{q}-\frac{\alpha}{q})}}\prod\limits_{i=1}^m \|f_i\|_{L^{q_i}_{\omega_i}(B_{R-\gamma })}.
\end{align}
Now, we need to prove that
\begin{align}\label{B-gamma+r}
\dfrac{1}{\omega(B_R)^{1/q+\lambda}}=\mathcal K.\prod\limits_{i=1}^m \dfrac{p^{-\gamma(\alpha_i+n)(\frac{1}{q_i}+\lambda_i)}}{\omega_i(B_{R-\gamma })^{1/{q_i}+\lambda_i}},
\end{align}
where 
$$
\mathcal K=\frac{(1-p^{-(\alpha+n)})^{\frac{1}{q}+\lambda}\prod\limits_{i=1}^m(1-p^{-n})^{\lambda_i}}{(1-p^{-n})^{\lambda}\prod\limits_{i=1}^m(1-p^{-(\alpha_i+n)})^{\frac{1}{q_i}+\lambda_i}}.
$$
Indeed, we get
\begin{align}\label{omegaBr}
\omega(B_R)=\int_{B_R}|x|_p^{\alpha}dx=\sum\limits_{\theta\leq R}\int_{S_{\theta}}p^{\theta\alpha}dx= \sum\limits_{\theta\leq R} p^{\theta(\alpha+n)}(1-p^{-n})=\dfrac{p^{R(\alpha+n)}(1-p^{-n})}{1-p^{-(\alpha+n)}},
\end{align}
and $\omega_i(B_{R-\gamma })=\frac{p^{(R-\gamma)(\alpha_i+n)}(1-p^{-n})}{1-p^{-(\alpha_i+n)}}$. Thus, by (\ref{lambdaMorrey}), we complete the proof for (\ref{B-gamma+r}).
\\
Applying (\ref{Delr2}) and (\ref{B-gamma+r}), we imply
$$\Delta_R \leq (1-p^{-n})\mathcal K.\|\Omega\|_{L^{q'}(S_0)}\sum\limits_{\gamma\in\mathbb Z}\dfrac{\Phi(p^{\gamma})}{p^{\gamma(1+(n+\alpha)\lambda)}}\prod\limits_{i=1}^m \|f_i\|_{{\mathop B\limits^.}^{q_i,\lambda _i}_{\omega_i}(\mathbb Q^n_p)}.$$
Therefore, we have $\|{\mathcal H}^p_{\Phi,\Omega}(\vec f)\|_{{\mathop B\limits^.}^{q,\lambda}_{\omega}(\mathbb Q^n_p)}\leq (1-p^{-n})\mathcal K.\mathcal C_1\|\Omega\|_{L^{q'}(S_0)}\prod\limits_{i=1}^m \|f_i\|_{{\mathop B\limits^.}^{q_i,\lambda _i}_{\omega_i}(\mathbb Q^n_p)}.$
\\

To give the proof for necessary condition of the theorem, let us now choose
$$
f_i(x)=|x|_p^{(n+\alpha_i)\lambda_i}, \,\,\textit{\rm for all}\, i=1,...,m.
$$
Then, we get  
\begin{align}\label{normfi}
\|f_i\|^{q_i}_{L^{q_i}_{\omega_i}(B_R)}&=\int_{B_R}|x|_p^{(n+\alpha_i)\lambda_iq_i+\alpha_i}dx=\sum\limits_{\theta \leq R}\int_{S_\theta}p^{\theta((n+\alpha_i)\lambda_iq_i+\alpha_i)}dx\nonumber
\\
&= \sum\limits_{\theta \leq R} p^{\theta(n+\alpha_i)(\lambda_iq_i+1)}(1-p^{-n})= \frac{p^{R(n+\alpha_i)(\lambda_iq_i+1)}(1-p^{-n})}{1-p^{-(n+\alpha_i)(\lambda_iq_i+1)}}.
\end{align}
Thus, by the similar argument as (\ref{omegaBr}), it is not hard to show that 
\begin{align}\label{estimatefi}
\|f_i\|_{{\mathop B\limits^.}^{q_i,\lambda _i}_{\omega_i}(\mathbb Q^n_p)}=\dfrac{(1-p^{-(\alpha_i+n)})^{\frac{1}{q_i}+\lambda_i}}{(1-p^{-n})^{\lambda_i}(1-p^{-(n+\alpha_i)(\lambda_iq_i+1)})^{\frac{1}{q_i}}},  \,\,\textit{\rm for all}\, i=1,...,m.
\end{align}
Now, $\|{\mathcal H}^p_{\Phi,\Omega}(\vec f)\|_{{\mathop B\limits^.}^{q,\lambda}_{\omega}(\mathbb Q^n_p)}$ is equal to
\begin{align}
&\mathop{\rm sup}\limits_{R\in\mathbb Z}\Big(\dfrac{1}{\omega(B_R)^{1+\lambda q}}\int_{B_R}\Big|\sum_{\gamma\in\mathbb Z}\int_{S_0} \dfrac{\Phi(p^{\gamma})}{p^{\gamma}}\Omega(y)\prod\limits_{i=1}^{m} \left |p^{\gamma}|x|_p^{-1}y\right|_p^{(n+\alpha_i)\lambda_i}dy\Big|^q\omega(x)dx\Big)^{1/q}\nonumber
\\
&=\mathop{\rm sup}\limits_{R\in\mathbb Z}\Big(\dfrac{1}{\omega(B_R)^{1+\lambda q}}\int_{B_R}\Big|\sum_{\gamma\in\mathbb Z}\int_{S_0} \dfrac{\Phi(p^{\gamma})}{p^{\gamma(1+(n+\alpha)\lambda)}}\Omega(y)\left |x\right|_p^{(n+\alpha)\lambda}dy\Big|^q\omega(x)dx\Big)^{1/q}\nonumber
\\
&=\mathcal C_1.\Big|\int_{S_0}\Omega(y)dy\Big|.\mathop{\rm sup}\limits_{R\in\mathbb Z}\Big(\dfrac{1}{\omega(B_R)^{1+\lambda q}}\int_{B_R}\left |x\right|_p^{(n+\alpha)\lambda q+\alpha}dx\Big)^{1/q}\nonumber.
\end{align}
By estimating as (\ref{omegaBr}) and (\ref{normfi}), we obtain
$$
\Big(\dfrac{1}{\omega(B_R)^{1+\lambda q}}\int_{B_R}\left |x\right|_p^{(n+\alpha)\lambda q+\alpha}dx\Big)^{\frac{1}{q}}=\dfrac{(1-p^{-(\alpha+n)})^{\frac{1}{q}+\lambda}}{(1-p^{-n})^{\lambda}(1-p^{-(n+\alpha)(\lambda q+1)})^{\frac{1}{q}}}.
$$
Thus, it follows from (\ref{estimatefi}) that
$$
\|{\mathcal H}^p_{\Phi,\Omega}(\vec f)\|_{{\mathop B\limits^.}^{q,\lambda}_{\omega}(\mathbb Q^n_p)}\geq\mathcal K.\mathcal C_1.\dfrac{\prod\limits_{i=1}^m(1-p^{-(n+\alpha_i)(\lambda_iq_i+1)})^{\frac{1}{q_i}}}{(1-p^{-(n+\alpha)(\lambda q +1)})^{\frac{1}{q}}}.\Big|\int_{S_0}\Omega(y)dy\Big|.\prod\limits_{i=1}^m \|f_i\|_{{\mathop B\limits^.}^{q_i,\lambda _i}_{\omega_i}(\mathbb Q^n_p)}.
$$
Since ${\mathcal H}^p_{\Phi,\Omega}$ is bounded from ${\mathop B\limits^.}^{q_1,\lambda _1}_{\omega_1}(\mathbb Q^n_p)\times\cdots \times {\mathop B\limits^.}^{{q_m},{\lambda _m}}_{\omega_m}(\mathbb Q^n_p)$ to ${\mathop B\limits^.}^{q,\lambda}_{\omega}(\mathbb Q^n_p)$, it immediately implies that $\mathcal C_1<\infty$. Thus, the proof of the theorem is finished.
\end{proof}
Next, we have the boundedness of $p$-adic rough multilinear Hausdorff operators on weighted $\lambda$-central Morrey $p$-adic spaces associated with the Muckenhoupt weights.
\begin{theorem}\label{TheoremMorrey1}
Let $1\leq q^*,\zeta<\infty$, $\Omega\in L^{(\frac{q}{\zeta})'}(S_0),$ $\lambda_i\in (\frac{-1}{q_i},0)$, for all $i=1,...,m$ and $\omega\in A_{\zeta}$ with the finite critical index $r_{\omega}$ for the reverse H\"{o}lder condition. Assume that $q>q^*\zeta r_{\omega}/(r_{\omega}-1)$, $\delta\in (1,r_\omega)$ and the two following conditions hold:
\begin{align}\label{lambdai*}
\lambda^*=\lambda_1+\cdots+\lambda_m,
\end{align}
\begin{align}\label{C2}
\mathcal C_2=\sum\limits_{\gamma\geq 0}\dfrac{\Phi(p^{\gamma})}{p^{\gamma(1+n\zeta\lambda^*)}}+\sum\limits_{\gamma<0}\dfrac{\Phi(p^{\gamma})}{p^{\gamma(1+n\lambda^*(\delta-1)/\delta))}}<\infty.
\end{align}
Then, ${\mathcal H}^p_{\Phi,\Omega}$ is bounded from ${\mathop B\limits^.}^{q_1,\lambda _1}_{\omega}(\mathbb Q^n_p)\times\cdots \times {\mathop B\limits^.}^{{q_m},{\lambda _m}}_{\omega}(\mathbb Q^n_p)$ to ${\mathop B\limits^.}^{q^*,\lambda^*}_{\omega}(\mathbb Q^n_p)$.
\end{theorem}
\begin{proof}
According to the Minkowski inequality, it implies that
\begin{align}\label{HfMor1}
\|{\mathcal H}^p_{\Phi,\Omega}(\vec f)\|_{L_\omega^{q^*}(B_R)} &\leq \sum\limits_{\gamma\in\mathbb Z}\dfrac{\Phi(p^{\gamma})}{p^{\gamma}}\Big(\int_{S_0}|\Omega(y)|\Big(\int_{B_R}\prod\limits_{i=1}^{m} |f_i(p^{\gamma}|x|_p^{-1}y)|^{q^*}\omega(x)dx\Big)^{1/{q^*}}dy\Big).
\end{align}
In view of the condition $q>q^*\zeta r_\omega/(r_\omega-1)$, there exists $r\in(1,r_\omega)$ such that $q/\zeta=q^*r'$. By  the H\"{o}lder inequality and the reverse H\"{o}lder condition, we have
\begin{align}\label{reverMor1}
\Big(\int_{B_R}\prod\limits_{i=1}^{m} |f_i(p^{\gamma}|x|_p^{-1}y)|^{q^*}\omega(x)dx\Big)^{1/{q^*}}&\leq \Big(\int_{B_R}\prod\limits_{i=1}^{m} |f_i(p^{\gamma}|x|_p^{-1}y)|^{\frac{q}{\zeta}}dx\Big)^{\zeta/{q}}\Big(\int_{B_R}\omega^r(x)dx\Big)^{1/{rq*}}\nonumber
\\
&\lesssim \Big(\int_{B_R}\prod\limits_{i=1}^{m} |f_i(p^{\gamma}|x|_p^{-1}y)|^{\frac{q}{\zeta}}dx\Big)^{\zeta/{q}}\omega^{\frac{1}{q*}}(B_R)|B_R|^{\frac{-\zeta}{q}}.
\end{align}
Next, by using the H\"{o}der inequality and estimating as (\ref{Holderqi}) and (\ref{S0Br}) above, we deduce
\begin{align}
\int_{S_0}|\Omega(y)|\Big(\int_{B_R}\prod\limits_{i=1}^{m} |f_i(p^{\gamma}|x|_p^{-1}y)|^{\frac{q}{\zeta}}dx\Big)^{\zeta/{q}}dy
&\leq \|\Omega\|_{L^{(\frac{q}{\zeta})'}(S_0)}\prod\limits_{i=1}^{m}\Big(\int_{S_0}\int_{B_R}|f_i(p^{\gamma}|x|_p^{-1}y)|^{\frac{q_i}{\zeta}}dxdy\Big)^{\zeta/q_i}.
\nonumber
\\
&\lesssim \|\Omega\|_{L^{(\frac{q}{\zeta})'}(S_0)}\prod\limits_{i=1}^m p^{\frac{\gamma\zeta n}{q_i}}\|f_i\|_{L^{q_i/\zeta}(B_{R-\gamma})}.\nonumber
\end{align}
Thus, by (\ref{HfMor1}) and (\ref{reverMor1}), we imply 
\begin{align}
\|{\mathcal H}^p_{\Phi,\Omega}(\vec f)\|_{L_\omega^{q^*}(B_R)} &\lesssim \omega^{\frac{1}{q*}}(B_R)|B_R|^{\frac{-\zeta}{q}}\|\Omega\|_{L^{(\frac{q}{\zeta})'}(S_0)}\sum\limits_{\gamma\in\mathbb Z}\dfrac{\Phi(p^{\gamma})}{p^{\gamma(1-\frac{\zeta n}{q})}}\prod\limits_{i=1}^m \|f_i\|_{L^{q_i/\zeta}(B_{R-\gamma})}.\nonumber
\end{align}
For $i=1,...,m$, by applying Proposition \ref{pro2.4DFan} again, we have
$$
\|f_i\|_{L^{q_i/\zeta}(B_{R-\gamma})}\lesssim |B_{R-\gamma}|^{\frac{\zeta}{q_i}}\omega^{\frac{-1}{q_i}}(B_{R-\gamma})\|f_i\|_{L_\omega^{q_i}(B_{R-\gamma})}.
$$
It follows that
\begin{align}\label{HfMor2}
\|{\mathcal H}^p_{\Phi,\Omega}(\vec f)\|_{L_\omega^{q^*}(B_R)}&
\lesssim\|\Omega\|_{L^{(\frac{q}{\zeta})'}(S_0)}\sum\limits_{\gamma\in\mathbb Z}\dfrac{\Phi(p^{\gamma})}{p^{\gamma(1-\frac{\zeta n}{q})}}\Big(\dfrac{|B_{R-\gamma}|}{|B_R|}\Big)^{\frac{\zeta}{q}}
\dfrac{\omega^{1/q^*}(B_R)}{\omega^{1/q}(B_{R-\gamma})}\prod\limits_{i=1}^m\|f_i\|_{L^{q_i}_\omega(B_{R-\gamma})}\nonumber
\\
&\lesssim \|\Omega\|_{L^{(\frac{q}{\zeta})'}(S_0)}\sum\limits_{\gamma\in\mathbb Z}\dfrac{\Phi(p^{\gamma})}{p^{\gamma}}.\dfrac{\omega^{1/q^*}(B_R)}{\omega^{1/q}(B_{R-\gamma})}\prod\limits_{i=1}^m\|f_i\|_{L^{q_i}_\omega(B_{R-\gamma})}.
\end{align}
Consequently, by (\ref{lambdai*}), we have $\dfrac{1}{\omega^{\frac{1}{q*}+\lambda*}(B_R)}\|{\mathcal H}^p_{\Phi,\Omega}(\vec f)\|_{L_\omega^{q^*}(B_R)}$ is controlled by
\begin{align}\label{HfMor3}
\|\Omega\|_{L^{(\frac{q}{\zeta})'}(S_0)}\prod\limits_{i=1}^m\|f_i\|_{{\mathop B\limits^.}^{q_i,\lambda_i}_{\omega}(\mathbb Q^n_p)}\times\Big(\sum\limits_{\gamma\in\mathbb Z}\dfrac{\Phi(p^{\gamma})}{p^{\gamma}}\Big(\dfrac{\omega(B_{R-\gamma})}{\omega(B_R)}\Big)^{\lambda^*}\Big).
\end{align}
Because of having (\ref{lambdai*}), it also gives $\lambda^*<0$. Thus, by $\delta\in(1,r_\omega)$ and Proposition \ref{pro2.3DFan}, if $\gamma\geq 0 $
\begin{align}\label{BRGa>=0}
\Big(\dfrac{\omega(B_{R-\gamma})}{\omega(B_R)}\Big)^{\lambda^*}\lesssim \Big(\dfrac{|B_{R-\gamma}|}{|B_R|}\Big)^{\zeta\lambda^*} \lesssim p^{-\gamma n\zeta\lambda^*},
\end{align}
and otherwise,
\begin{align}\label{BRGa<0}
\Big(\dfrac{\omega(B_{R-\gamma})}{\omega(B_R)}\Big)^{\lambda^*}\lesssim \Big(\dfrac{|B_{R-\gamma}|}{|B_R|}\Big)^{\lambda^*(\delta-1)/\delta} \lesssim p^{-\gamma n\lambda^*(\delta-1)/\delta)}.
\end{align}
By (\ref{HfMor3}), (\ref{BRGa>=0}) and (\ref{BRGa<0}), we obtain
$$
\|{\mathcal H}^p_{\Phi,\Omega}\|_{{\mathop B\limits^.}^{q^*,\lambda^*}_{\omega}(\mathbb Q^n_p)}\lesssim \mathcal C_2.\|\Omega\|_{L^{(\frac{q}{\zeta})'}(S_0)}\prod\limits_{i=1}^m\|f_i\|_{{\mathop B\limits^.}^{q_i,\lambda_i}_{\omega}(\mathbb Q^n_p)}.
$$
This completes the proof of theorem.
\end{proof}
\begin{theorem}\label{TheoremHerz}
Let $\omega_1(x)=|x|^{\alpha_1}_p,...,\omega_m(x)=|x|^{\alpha_m}_p$, $\omega(x)= |x|^\alpha_p$ and $\Omega\in L^{q'}(S_0)$. Then, we have that ${\mathcal H}^p_{\Phi,\Omega}$ is bounded from ${K}^{\beta_1,\ell_1}_{q_1,\omega_1}(\mathbb Q^n_p)\times\cdots \times {K}^{{\beta_m, \ell_m}}_{q_m,\omega_m}(\mathbb Q^n_p)$ to ${K}^{\beta,\ell}_{q, \omega}(\mathbb Q^n_p)$ if and only if 
$$
\mathcal C_{3}= \sum\limits_{\gamma\in\mathbb Z} {\dfrac{\Phi(p^\gamma)}{p^{\gamma(1-\beta-\frac{n+\alpha}{q})}}} <  + \infty.
$$
Furthermore, ${\big\|{\mathcal H}^p_{\Phi,\Omega}\big\|_{{K}^{\beta_1,\ell_1}_{q_1,\omega_1}(\mathbb Q^n_p)\times\cdots \times {K}^{{\beta_m, \ell_m}}_{q_m,\omega_m}(\mathbb Q^n_p)\to{K}^{\beta,\ell}_{q, \omega}(\mathbb Q^n_p)}\simeq \mathcal C_{3}}.$
\end{theorem}
\begin{proof}
We begin with the proof for the sufficient condition of theorem. For $k\in\mathbb Z$, by using Minkowski inequality and H\"{o}lder inequality again and making as (\ref{Delr1}) above,  we have
$$
\|{\mathcal H}^p_{\Phi,\Omega}(\vec f)\chi_k\|_{L^q_\omega(Q_p^n)}\leq \|\Omega\|_{L^{q'}(S_0)}\sum\limits_{\gamma\in\mathbb Z}\dfrac{\Phi(p^{\gamma})}{p^{\gamma}}\prod\limits_{i=1}^{m}\Big(\int_{S_0}\int_{S_k}|f_i(p^{\gamma}|x|_p^{-1}y)|^{q_i}\omega_i(x)dxdy\Big)^{1/q_i}.
$$
Next, a simple calculation as (\ref{S0Br}) above follows us to have that
\begin{align}\label{S0Sr}
&\int_{S_0}\int_{S_k}|f_i(p^{\gamma}|x|_p^{-1}y)|^{q_i}\omega_i(x)dxdy =\int_{S_0}\int_{S_k}|f_i(p^{\gamma-k}y)|^{q_i}p^{k\alpha_i}dxdy\nonumber
\\
&\simeq \int_{S_0}|f_i(p^{\gamma-k}y)|^{q_i}p^{k\alpha_i}p^{nk}dy= p^{\gamma(n+\alpha_i)}\int_{S_0}|f_i(p^{-(-\gamma+k)}y)|^{q_i}p^{(-\gamma+k)\alpha_i}p^{n(-\gamma+k)}dy\nonumber
\\
&= p^{\gamma(n+\alpha_i)}\int_{S_{-\gamma+k}}|f_i(z)|^{q_i}\omega_i(z)dz= p^{\gamma(n+\alpha_i)}\|f_i\chi_{-\gamma+k}\|^{q_i}_{L^{q_i}_{\omega_i}(\mathbb Q_p^n)}.
\end{align}
Thus, we have
\begin{align}\label{HfChik}
\|{\mathcal H}^p_{\Phi,\Omega}(\vec f)\chi_k\|_{L^q_\omega(Q_p^n)}\lesssim \|\Omega\|_{L^{q'}(S_0)}\sum\limits_{\gamma\in\mathbb Z}\dfrac{\Phi(p^{\gamma})}{p^{\gamma(1-\frac{n+\alpha}{q})}}\prod\limits_{i=1}^{m}\|f_i\chi_{-\gamma+k}\|_{L^{q_i}_{\omega_i}(\mathbb Q_p^n)}.
\end{align}
Combining this with the Minkowski inequality, we have
\begin{align}
\|{\mathcal H}^p_{\Phi,\Omega}(\vec f)\|_{K^{\beta,\ell}_{q,\omega}(\mathbb Q_p^n)}&= \Big(\sum\limits_{k=-\infty}^{\infty}p^{k\beta\ell}\|{\mathcal H}^p_{\Phi,\Omega}(\vec f)\chi_k\|_{L^q_\omega(Q_p^n)}^\ell\Big)^{1/\ell}\nonumber
\\
&\lesssim\|\Omega\|_{L^{q'}(S_0)}\Big(\sum\limits_{k=-\infty}^{\infty}\Big(\sum\limits_{\gamma\in\mathbb Z}\dfrac{\Phi(p^{\gamma})}{p^{\gamma(1-\frac{n+\alpha}{q})}}p^{k\beta}\prod\limits_{i=1}^{m}\|f_i\chi_{-\gamma+k}\|_{L^{q_i}_{\omega_i}(\mathbb Q_p^n)}\Big)^\ell\Big)^{1/\ell}\nonumber
\\
&\lesssim\|\Omega\|_{L^{q'}(S_0)}\sum\limits_{\gamma\in\mathbb Z}\dfrac{\Phi(p^{\gamma})}{p^{\gamma(1-\frac{n+\alpha}{q})}}\Big(\sum\limits_{k=-\infty}^{\infty}p^{k\beta\ell}\prod\limits_{i=1}^{m}\|f_i\chi_{-\gamma+k}\|^\ell_{L^{q_i}_{\omega_i}(\mathbb Q_p^n)}\Big)^{1/\ell}.
\nonumber
\end{align}
In view of  the H\"{o}lder inequality and the definition of the Herz space, we obtain that
\begin{align}
&\Big(\sum\limits_{k=-\infty}^{\infty}p^{k\beta\ell}\prod\limits_{i=1}^{m}\|f_i\chi_{-\gamma+k}\|^\ell_{L^{q_i}_{\omega_i}(\mathbb Q_p^n)}\Big)^{1/\ell}\nonumber
\\
&\,\,\,\leq \prod\limits_{i=1}^m\Big(\sum\limits_{k=-\infty}^{\infty}p^{k\beta_i\ell_i}\|f_i\chi_{-\gamma+k}\|^{\ell_i}_{L^{q_i}_{\omega_i}(\mathbb Q_p^n)}\Big)^{1/\ell_i}=\prod\limits_{i=1}^{m}p^{\gamma\beta}\|f_i\|_{K^{\beta_i,\ell_i}_{q_i,\omega_i}(\mathbb Q^n_p)},
\nonumber
\end{align}
which implies that
$\|{\mathcal H}^p_{\Phi,\Omega}(\vec f)\|_{K^{\beta,\ell}_{q,\omega}(\mathbb Q_p^n)}\lesssim \mathcal  C_2.\|\Omega\|_{L^{q'}(S_0)}.\prod\limits_{i=1}^{m}\|f_i\|_{K^{\beta_i,\ell_i}_{q_i,\omega_i}(\mathbb Q^n_p)}.$

Conversely, suppose that  ${\mathcal H}^p_{\Phi,\Omega}$ is bounded from ${K}^{\beta_1,\ell_1}_{q_1,\omega_1}(\mathbb Q^n_p)\times\cdots \times {K}^{{\beta_m, \ell_m}}_{q_m,\omega_m}(\mathbb Q^n_p)$ to ${K}^{\beta,\ell}_{q, \omega}(\mathbb Q^n_p)$. For $i=1,...,m$ and $r\in\mathbb Z^+$, let us choose the following functions
$$
f_{i,r}(x) = \left\{ \begin{array}{l}
0\,\,\,\,\,\,\,\,\,\,\,\,\,\,\,\,\,\,\,\,\,\,\,\,\,\,\,\,\,\,\,\,\,\,,|x|_p<1,
\\
|x|_p^{-\beta_i-\frac{n+\alpha_i}{q_i}-\frac{1}{p^r}}, |x|_p\geq 1.
\end{array} \right.
$$
It is clear to see that when $k<0$ then $f_{i,r}\chi_k=0$. Otherwise, we have
\begin{align}\label{fir}
\|f_{i,r}\chi_k\|_{L^{q_i}_{\omega_i}(\mathbb Q^n_p)}^{q_i}=\int_{S_k}|x|^{-\beta_iq_i-n-\frac{q_i}{p^r}}dx \simeq p^{k(-\beta_iq_i-n-\frac{q_i}{p^r})}p^{kn}=p^{-kq_i(\beta_i+\frac{1}{p^r})}.
\end{align}
Thus, it leads that $\|f_{i,r}\|_{K^{\beta_i,\ell_i}_{q_i,\omega_i}(\mathbb Q^n_p)}>0$. Moreover,
\begin{align}\label{estimatefir}
\|f_{i,r}\|_{K^{\beta_i,\ell_i}_{q_i,\omega_i}(\mathbb Q^n_p)}&=\Big(\sum\limits_{k=-\infty}^{\infty}p^{k\beta_i\ell_i}\|f_{i,r}\chi_k\|^{\ell_i}_{L^{q_i}_{\omega_i}(\mathbb Q_p^n)}\Big)^{1/\ell_i}
\nonumber
\\
&\simeq \Big(\sum\limits_{k=0}^{\infty}p^{\frac{-k\ell_i}{p^r}}\Big)^{1/\ell_i} = \dfrac{1}{\Big(1-p^{\frac{-\ell_i}{p^r}}\Big)^{1/\ell_i}}<\infty.
\end{align}
Now, we get
\begin{align}
\left|{\mathcal H}^p_{\Phi,\Omega}(\vec f)(x)\right|&=\Big|\sum\limits_{\gamma\in\mathbb Z}\int_{S_0}\frac{\Phi(p^{-\gamma})}{p^{-\gamma}}\Omega(y)\prod\limits_{i=1}^{m}f_{i,r}(p^{-\gamma}|x|_p^{-1}y)dy\Big|\nonumber
\\
&=\Big|\int_{S_0}\Omega(y)dy\Big|.\Big|\sum\limits_{\gamma\geq -{\rm log}_p|x|_p}\dfrac{\Phi(p^{-\gamma})}{p^{-\gamma(1-\beta-\frac{n+\alpha}{q}-\frac{m}{p^r})}}\Big|.|x|_p^{-(\beta+\frac{n+\alpha}{q}+\frac{m}{p^r})}.\nonumber
\\
&\geq \Big|\int_{S_0}\Omega(y)dy\Big|.\Big|\sum\limits_{\gamma\geq-r}\dfrac{\Phi(p^{-\gamma})}{p^{-\gamma(1-\beta-\frac{n+\alpha}{q}-\frac{m}{p^r})}}\Big|.|x|_p^{-(\beta+\frac{n+\alpha}{q}+\frac{m}{p^r})}\chi_{\mathbb Q^n_p\setminus B_{r-1}}(x).\nonumber
\end{align}
Therefore, for $r\in\mathbb Z$, it follows immediately that
\begin{align}
&\|{\mathcal H}^p_{\Phi,\Omega}(\vec f)\|_{K^{\beta,\ell}_{q,\omega}(\mathbb Q^n_p)}\geq \Big(\sum\limits_{k=r}^{\infty}p^{k\beta\ell}\|{\mathcal H}^p_{\Phi,\Omega}(\vec f)\chi_k\|_{L^q_\omega(\mathbb Q^n_p)}^{\ell}\Big)^{1/\ell}\nonumber
\\
&\,\,\,\,\,\,\,\,\,\,\,\geq \Big|\int_{S_0}\Omega(y)dy \Big|\sum\limits_{\gamma\geq -r}\dfrac{\Phi(p^{-\gamma})}{p^{-\gamma(1-\beta-\frac{n+\alpha}{q}-\frac{m}{p^r})}}.\Big(\sum\limits_{k=r}^{\infty}p^{k\beta\ell}\left\||x|_p^{-(\beta+\frac{n+\alpha}{q}+\frac{m}{p^r})}\chi_k\right\|_{L^q_\omega(\mathbb Q^n_p)}^{\ell}\Big)^{1/\ell}\nonumber
\\
&\,\,\,\,\,\,\,\,\,\,\,=:\Big|\int_{S_0}\Omega(y)dy \Big|\sum\limits_{\gamma\geq -r}\dfrac{\Phi(p^{-\gamma})}{p^{-\gamma(1-\beta-\frac{n+\alpha}{q}-\frac{m}{p^r})}}.\Big(\sum\limits_{k=r}^{\infty}p^{k\beta\ell}\left\|g_k\right\|_{L^q_\omega(\mathbb Q^n_p)}^{\ell}\Big)^{1/\ell}\nonumber.
\end{align}
Estimating as (\ref{fir}) above, it implies $\left\|g_k\right\|_{L^q_\omega(\mathbb Q^n_p)}\simeq p^{-k(\beta+\frac{m}{p^r})}$. Thus, 
\begin{align}
\|{\mathcal H}^p_{\Phi,\Omega}(\vec f)\|_{K^{\beta,\ell}_{q,\omega}(\mathbb Q^n_p)} &\gtrsim \Big|\int_{S_0}\Omega(y)dy \Big|\sum\limits_{\gamma\geq -r}\dfrac{\Phi(p^{-\gamma})}{p^{-\gamma(1-\beta-\frac{n+\alpha}{q}-\frac{m}{p^r})}}.\Big(\sum\limits_{k=r}^{\infty}p^\frac{-km\ell}{p^r}\Big)^{1/\ell}\nonumber
\\
&=\Big|\int_{S_0}\Omega(y)dy \Big|\sum\limits_{\gamma\geq -r}\dfrac{\Phi(p^{-\gamma})}{p^{-\gamma(1-\beta-\frac{n+\alpha}{q}-\frac{m}{p^r})}}\frac{p^{\frac{-rm}{p^r}}}{(1-p^{\frac{-m\ell}{p^r}})^{1/\ell}}.
\nonumber
\end{align}
Hence, by (\ref{estimatefir}), we have
$$
\|{\mathcal H}^p_{\Phi,\Omega}(\vec f)\|_{K^{\beta,\ell}_{q,\omega}(\mathbb Q^n_p)} \gtrsim \mathcal A(r).\Big|\int_{S_0}\Omega(y)dy \Big|\sum\limits_{\gamma\geq -r}\dfrac{\Phi(p^{-\gamma})}{p^{-\gamma(1-\beta-\frac{n+\alpha}{q}-\frac{m}{p^r})}}\prod\limits_{i=1}^{m}\|f_{i,r}\|_{K^{\beta_i,\ell_i}_{q_i,\omega_i}(\mathbb Q^n_p)},
$$
where $\mathcal A(r)=\frac{p^{\frac{-rm}{p^r}}\prod\limits_{i=1}^{m}(1-p^{\frac{-\ell_i}{p^r}})^{1/\ell_i}}{(1-p^{\frac{-m\ell}{p^r}})^{1/\ell}}.$ Note that, by lettting $r$ big enough, we get 
$$
p^{\frac{-\gamma m}{p^r}}\lesssim 1.
$$
By having $\frac{1}{\ell_1}+\cdots+\frac{1}{\ell_m}=\frac{1}{\ell}$, and lettting $r\to +\infty$, we obtain
$$
\mathop{\rm lim}\limits_{r\to +\infty}\mathcal A(r) = a>0.
$$
Hence, by the dominated convergence theorem of Lebesgue and the boundedness of ${\mathcal H}^p_{\Phi,\Omega}$ on  ${K}^{\beta_1,\ell_1}_{q_1,\omega_1}(\mathbb Q^n_p)\times\cdots \times {K}^{{\beta_m, \ell_m}}_{q_m,\omega_m}(\mathbb Q^n_p)$ to ${K}^{\beta,\ell}_{q, \omega}(\mathbb Q^n_p)$, we get
$$
\sum\limits_{\gamma\in\mathbb Z}\dfrac{\Phi(p^{-\gamma})}{p^{-\gamma(1-\beta-\frac{n+\alpha}{q})}}<  + \infty,
$$
which finishes the proof of this theorem.
\end{proof}
\begin{theorem}\label{TheoremHerz1}
Let $1\leq q^*<\infty$, $\beta^*\in\mathbb R$, $\beta_1,...,\beta_m\in\mathbb R^{-}$ (set of all negative real numbers) such that 
$$\dfrac{1}{q^*}+\dfrac{\beta^*}{n}=\dfrac{1}{q}+\dfrac{\beta}{n}.
$$
Suppose that $\omega\in A_{\zeta}$, $1\leq \zeta<\infty$, with the finite critical index $r_{\omega}$ for the reverse H\"{o}lder condition and $q>q^*\zeta r_{\omega}/(r_{\omega}-1)$, $\Omega\in L^{(\frac{q}{\zeta})'}(S_0),\delta\in (1,r_\omega)$.
\\
$(\rm i)$ If $\dfrac{1}{q_i}+\dfrac{\alpha_i}{n}\geq 0$, for $i=1,...,m$, and 
$$\mathcal C_{4.1}= \sum\limits_{\gamma\geq 0}\dfrac{\Phi(p^{\gamma})}{p^{\gamma(1-\frac{\zeta n}{q}-\zeta\beta)}}+ \sum\limits_{\gamma< 0}\dfrac{\Phi(p^{\gamma})}{p^{\gamma(1-n\frac{(\delta-1)}{\delta}(\frac{1}{q}+\frac{\beta}{n}))}}<\infty,
$$
then ${\mathcal H}^p_{\Phi,\Omega}$ is bounded from ${\mathop{K}\limits^.}^{\beta_1,\ell_1}_{q_1,\omega}(\mathbb Q^n_p)\times\cdots \times {\mathop{K}\limits^.}^{\beta_m,\ell_m}_{q_m,\omega}(\mathbb Q^n_p)$ to ${\mathop{K}\limits^.}^{\beta^*,\ell}_{q^*,\omega}(\mathbb Q^n_p)$.

$\,(\rm ii)$ If $\dfrac{1}{q_i}+\dfrac{\alpha_i}{n}< 0$, for $i=1,...,m$, and 
$$\mathcal C_{4.2}= \sum\limits_{\gamma< 0}\dfrac{\Phi(p^{\gamma})}{p^{\gamma(1-\frac{\zeta n}{q}+\zeta\beta)}}+ \sum\limits_{\gamma\geq  0}\dfrac{\Phi(p^{\gamma})}{p^{\gamma(1-n\frac{(\delta-1)}{\delta}(\frac{1}{q}+\frac{\beta}{n}))}}<\infty,
$$
then ${\mathcal H}^p_{\Phi,\Omega}$ is bounded from ${\mathop{K}\limits^.}^{\beta_1,\ell_1}_{q_1,\omega}(\mathbb Q^n_p)\times\cdots \times {\mathop{K}\limits^.}^{\beta_m,\ell_m}_{q_m,\omega}(\mathbb Q^n_p)$ to ${\mathop{K}\limits^.}^{\beta^*,\ell}_{q^*,\omega}(\mathbb Q^n_p).$
\end{theorem}
\begin{proof}
According to estimating as (\ref{HfMor2}), we also have
\begin{align}
 \|{\mathcal H}^p_{\Phi,\Omega}(\vec f)\chi_k\|_{L_\omega^{q^*}(\mathbb Q^n_p)}&\lesssim \|\Omega\|_{L^{(\frac{q}{\zeta})'}(S_0)}\sum\limits_{\gamma\in\mathbb Z}\dfrac{\Phi(p^{\gamma})}{p^{\gamma}}.\dfrac{\omega^{1/q^*}(B_k)}{\omega^{1/q}(B_{-\gamma+k})}\prod\limits_{i=1}^m\|f_i\|_{L^{q_i}_\omega(B_{-\gamma+k})}\nonumber.
\end{align}
From this, by the Minkowski inequality, $\|{\mathcal H}^p_{\Phi,\Omega}(\vec f)\|_{{\mathop{K}\limits^.}^{\beta^*,\ell}_{q^*,\omega}(\mathbb Q^n_p)}$ is controlled by
\begin{align}\label{HfHerz12}
&\|\Omega\|_{L^{(\frac{q}{\zeta})'}(S_0)}\Big(\sum\limits_{k=-\infty}^{\infty}\Big(\sum\limits_{\gamma\in\mathbb Z}\dfrac{\Phi(p^{\gamma})}{p^{\gamma}}.\dfrac{\omega^{\frac{1}{q^*}+\frac{\beta^*}{n}}(B_k)}{\omega^{\frac{1}{q}}(B_{-\gamma+k})}\prod\limits_{i=1}^m\|f_i\|_{L^{q_i}_\omega(B_{-\gamma+k})}\Big)^\ell\Big)^{1/\ell}
\nonumber
\\
&\lesssim \|\Omega\|_{L^{(\frac{q}{\zeta})'}(S_0)}\sum\limits_{\gamma\in\mathbb Z}\dfrac{\Phi(p^{\gamma})}{p^{\gamma}}\Big(\sum\limits_{k=-\infty}^{\infty}\Big(\dfrac{\omega^{\frac{1}{q^*}+\frac{\beta^*}{n}}(B_k)}{\omega^{\frac{1}{q}}(B_{-\gamma+k})}\prod\limits_{i=1}^m\|f_i\|_{L^{q_i}_\omega(B_{-\gamma+k})}\Big)^{\ell}\Big)^{1/\ell}.
\end{align}
By the relation $\frac{1}{q^*}+\frac{\beta^*}{n}= \frac{1}{q}+\frac{\beta}{n}$ and the H\"{o}lder inequality,  we get
\begin{align}
&\Big(\sum\limits_{k=-\infty}^{\infty}\Big(\dfrac{\omega^{\frac{1}{q^*}+\frac{\beta^*}{n}}(B_k)}{\omega^{\frac{1}{q}}(B_{-\gamma+k})}\prod\limits_{i=1}^m\|f_i\|_{L^{q_i}_\omega(B_{-\gamma+k})}\Big)^{\ell}\Big)^{1/\ell}\nonumber
\\
&\,\,\,\,\,\,\,\,\,\,\,\,\,\,\,\,\,\,\,\,\leq \prod\limits_{i=1}^m\Big(\sum\limits_{k=-\infty}^{\infty}\Big(\dfrac{\omega^{\frac{1}{q_i}+\frac{\beta_i}{n}}(B_k)}{\omega^{\frac{1}{q_i}}(B_{-\gamma+k})}\|f_i\|_{L^{q_i}_\omega(B_{-\gamma+k})}\Big)^{\ell_i}\Big)^{1/\ell_i}.\nonumber
\end{align}
Thus, by (\ref{HfHerz12}), it follows that
\begin{align}\label{HfHerz13}
&\|{\mathcal H}^p_{\Phi,\Omega}(\vec f)\|_{{\mathop{K}\limits^.}^{\beta^*,\ell}_{q^*,\omega}(\mathbb Q^n_p)}\lesssim \|\Omega\|_{L^{(\frac{q}{\zeta})'}(S_0)}\sum\limits_{\gamma\in\mathbb Z}\dfrac{\Phi(p^{\gamma})}{p^{\gamma}}\prod\limits_{i=1}^m\Big(\sum\limits_{k=-\infty}^{\infty}\Big(\dfrac{\omega^{\frac{1}{q_i}+\frac{\beta_i}{n}}(B_k)}{\omega^{\frac{1}{q_i}}(B_{-\gamma+k})}\|f_i\|_{L^{q_i}_\omega(B_{-\gamma+k})}\Big)^{\ell_i}\Big)^{1/\ell_i}\nonumber
\\
&\lesssim \|\Omega\|_{L^{(\frac{q}{\zeta})'}(S_0)}\sum\limits_{\gamma\in\mathbb Z}\dfrac{\Phi(p^{\gamma})}{p^{\gamma}}\times\nonumber
\\
&\times\prod\limits_{i=1}^m\Big(\sum\limits_{k=-\infty}^{\infty}\Big(\dfrac{\omega^{\frac{1}{q_i}+\frac{\beta_i}{n}}(B_k)}{\omega^{\frac{1}{q_i}+\frac{\beta_i}{n}}(B_{-\gamma+k})}\sum\limits_{\eta=-\infty}^{-\gamma}\Big(\dfrac{\omega(B_{-\gamma+k})}{\omega(B_{\eta+k})}\Big)^{\frac{\beta_i}{n}}\omega^{\frac{\beta_i}{n}}(B_{\eta+k})\|f_i\chi_{\eta+k}\|_{L^{q_i}_\omega(\mathbb Q^n_p)}\Big)^{\ell_i}\Big)^{1/\ell_i}.
\end{align}
From $\beta_i<0$ and $\eta\leq-\gamma$, by Proposition \ref{pro2.3DFan} again, we see that
\begin{align}\label{omegaHez11}
\Big(\dfrac{\omega(B_{-\gamma+k})}{\omega(B_{\eta+k})}\Big)^{\frac{\beta_i}{n}}\lesssim \Big(\dfrac{|B_{-\gamma+k}|}{|B_{\eta +k}|}\Big)^{\frac{\beta_i(\delta-1)}{n\delta}}\lesssim p^{(-\gamma-\eta)\beta_i(\delta-1)/\delta}.
\end{align}
When $\dfrac{1}{q_i}+\dfrac{\beta_i}{n}\geq 0$, Proposition \ref{pro2.3DFan} follows us to have
\begin{align}\label{omegaHez12}
\Big(\dfrac{\omega(B_k)}{\omega(B_{-\gamma+k})}\Big)^{\frac{1}{q_i}+\frac{\beta_i}{n}} \lesssim \left\{ \begin{array}{l}
\Big(\dfrac{|B_k|}{|B_{-\gamma+k}|}\Big)^{\zeta(\frac{1}{q_i}+\frac{\beta_i}{n})}\lesssim p^{\gamma n\zeta(\frac{1}{q_i}+\frac{\beta_i}{n})},\,\,\,\,\,\,\,\,\,\,\,\,\,\textit{\rm if}\,\gamma\geq 0,
\\
\\
\Big(\dfrac{|B_k|}{|B_{-\gamma+k}|}\Big)^{\frac{\delta-1}{\delta}(\frac{1}{q_i}+\frac{\beta_i}{n})}\lesssim p^{\gamma n\frac{\delta-1}{\delta}(\frac{1}{q_i}+\frac{\beta_i}{n})},\,\,\textit{\rm otherwise}.
\end{array} \right.
\end{align}
When $\dfrac{1}{q_i}+\dfrac{\beta_i}{n}< 0$, it is easy to see that Proposition \ref{pro2.3DFan} yields 
\begin{align}\label{omegaHerz13}
\Big(\dfrac{\omega(B_k)}{\omega(B_{-\gamma+k})}\Big)^{\frac{1}{q_i}+\frac{\beta_i}{n}} \lesssim \left\{ \begin{array}{l}
\Big(\dfrac{|B_k|}{|B_{-\gamma+k}|}\Big)^{\zeta(\frac{1}{q_i}+\frac{\beta_i}{n})}\lesssim p^{\gamma n\zeta(\frac{1}{q_i}+\frac{\beta_i}{n})},\,\,\,\,\,\,\,\,\,\,\,\,\,\textit{\rm if}\,\gamma < 0,
\\
\\
\Big(\dfrac{|B_k|}{|B_{-\gamma+k}|}\Big)^{\frac{\delta-1}{\delta}(\frac{1}{q_i}+\frac{\beta_i}{n})}\lesssim p^{\gamma n\frac{\delta-1}{\delta}(\frac{1}{q_i}+\frac{\beta_i}{n})},\,\,\textit{\rm otherwise}.
\end{array} \right.
\end{align}
In the case $(i)$, by (\ref{HfHerz13}), (\ref{omegaHez11}) and (\ref{omegaHez12}), we obtain that
\begin{align}\label{HfHerz14}
&\|{\mathcal H}^p_{\Phi,\Omega}(\vec f)\|_{{\mathop{K}\limits^.}^{\beta^*,\ell}_{q^*,\omega}(\mathbb Q^n_p)}\nonumber
\lesssim \|\Omega\|_{L^{(\frac{q}{\zeta})'}(S_0)}\sum\limits_{\gamma\geq 0}\dfrac{\Phi(p^{\gamma})}{p^{\gamma(1-\frac{\zeta n}{q}-\zeta\beta)}}\times\nonumber
\\
&\,\,\,\times\prod\limits_{i=1}^m\Big(\sum\limits_{k=-\infty}^{\infty}\Big(\sum\limits_{\eta=-\infty}^{-\gamma}p^{(-\gamma-\eta)\beta_i(\delta-1)/\delta}\omega^{\frac{\beta_i}{n}}(B_{\eta+k})\|f_i\chi_{\eta+k}\|_{L^{q_i}_\omega(\mathbb Q^n_p)}\Big)^{\ell_i}\Big)^{1/\ell_i}\nonumber
\\
&\,\,\,+\|\Omega\|_{L^{(\frac{q}{\zeta})'}(S_0)}\sum\limits_{\gamma< 0}\dfrac{\Phi(p^{\gamma})}{p^{\gamma(1-n\frac{(\delta-1)}{\delta}(\frac{1}{q}+\frac{\beta}{n}))}}\times\nonumber
\\
&\,\,\,\times\prod\limits_{i=1}^m\Big(\sum\limits_{k=-\infty}^{\infty}\Big(\sum\limits_{\eta=-\infty}^{-\gamma}p^{(-\gamma-\eta)\beta_i(\delta-1)/\delta}\omega^{\frac{\beta_i}{n}}(B_{\eta+k})\|f_i\chi_{\eta+k}\|_{L^{q_i}_\omega(\mathbb Q^n_p)}\Big)^{\ell_i}\Big)^{1/\ell_i}.
\end{align} 
Note that, by using the Minkowski inequality again, we deduce
\begin{align}
&\Big(\sum\limits_{k=-\infty}^{\infty}\Big(\sum\limits_{\eta=-\infty}^{-\gamma}p^{(-\gamma-\eta)\beta_i(\delta-1)/\delta}\omega^{\frac{\beta_i}{n}}(B_{\eta+k})\|f_i\chi_{\eta+k}\|_{L^{q_i}_\omega(\mathbb Q^n_p)}\Big)^{\ell_i}\Big)^{1/\ell_i}
\nonumber
\\
&\leq \sum\limits_{\eta=-\infty}^{-\gamma}p^{(-\gamma-\eta)\beta_i(\delta-1)/\delta}\Big(\sum\limits_{k=-\infty}^{\infty}\Big(\omega^{\frac{\beta_i}{n}}(B_{\eta+k})\|f_i\chi_{\eta+k}\|_{L^{q_i}_\omega(\mathbb Q^n_p)}\Big)^{\ell_i}\Big)^{1/\ell_i}
\nonumber 
\\
&\leq \sum\limits_{\eta=-\infty}^{-\gamma}p^{(-\gamma-\eta)\beta_i(\delta-1)/\delta}\|f_i\|_{{\mathop{K}\limits^.}^{\beta_i,\ell_i}_{q_i,\omega}(\mathbb Q^n_p)}\lesssim \|f_i\|_{{\mathop{K}\limits^.}^{\beta_i,\ell_i}_{q_i,\omega}(\mathbb Q^n_p)}.
\nonumber
\end{align}
Therefore, by (\ref{HfHerz14}), we have
\begin{align}
\|{\mathcal H}^p_{\Phi,\Omega}(\vec f)\|_{{\mathop{K}\limits^.}^{\beta^*,\ell}_{q^*,\omega}(\mathbb Q^n_p)}\nonumber
\lesssim \mathcal C_{4.1}.\|\Omega\|_{L^{(\frac{q}{\zeta})'}(S_0)}\prod\limits_{i=1}^m  \|f_i\|_{{\mathop{K}\limits^.}^{\beta_i,\ell_i}_{q_i,\omega}(\mathbb Q^n_p)},
\end{align} 
which completes the proof for this case.
\\
In this case $(\rm ii)$, from by (\ref{HfHerz13}), (\ref{omegaHez11}), (\ref{omegaHerz13}) and a similar argument as above, we also have
\begin{align}
\|{\mathcal H}^p_{\Phi,\Omega}(\vec f)\|_{{\mathop{K}\limits^.}^{\beta^*,\ell}_{q^*,\omega}(\mathbb Q^n_p)}\nonumber
\lesssim \mathcal C_{4.2}.\|\Omega\|_{L^{(\frac{q}{\zeta})'}(S_0)}\prod\limits_{i=1}^m  \|f_i\|_{{\mathop{K}\limits^.}^{\beta_i,\ell_i}_{q_i,\omega}(\mathbb Q^n_p)}.
\end{align} 
This implies that the proof of theorem is finished.
\end{proof}
\begin{theorem}\label{TheoremMorreyHerz}
Let $\omega_i,\omega, \Omega$ be as Theorem \ref{TheoremHerz}, $\lambda_1, ..., \lambda_m >0$ and  the hypothesis (\ref{lambdai*}) in Theorem \ref{TheoremMorrey1} holds.
Then, ${\mathcal H}^p_{\Phi,\Omega}$ is  a bounded operator from ${MK}^{\beta_1,\lambda_1}_{\ell_1, q_1,\omega_1}(\mathbb Q^n_p)\times\cdots \times {MK}^{{\beta_m, \lambda_m}}_{\ell_m, q_m,\omega_m}(\mathbb Q^n_p)$ to ${MK}^{\beta,\lambda^*}_{\ell,q, \omega}(\mathbb Q^n_p)$ if and only if 
$$
\mathcal C_{5}= \sum\limits_{\gamma\in\mathbb Z} {\dfrac{\Phi(p^\gamma)}{p^{\gamma(1-\beta-\frac{n+\alpha}{q}+\lambda^*)}}} <  + \infty.
$$
Moreover, ${\big\|{\mathcal H}^p_{\Phi,\Omega}\big\|_{{MK}^{\beta_1,\lambda_1}_{\ell_1, q_1,\omega_1}(\mathbb Q^n_p)\times\cdots \times {MK}^{{\beta_m, \lambda_m}}_{\ell_m, q_m,\omega_m}(\mathbb Q^n_p)\to {MK}^{\beta,\lambda^*}_{\ell,q, \omega}(\mathbb Q^n_p)}\simeq \mathcal C_{5}}.$
\end{theorem}
\begin{proof}
By the estimation (\ref{HfChik}) and the definition of the Morrey-Herz $p$-adic space, we have
\begin{align}
&\big\|{\mathcal H}^p_{\Phi,\Omega}(\vec f)\big\|_{{MK}^{\beta,\lambda^*}_{\ell, q,\omega}(\mathbb Q^n_p)}=\mathop{\rm sup}\limits_{k_0\in\mathbb Z}p^{-k_0\lambda^*}\Big(\sum\limits_{k=-\infty}^{k_0}p^{k\beta\ell}\|{\mathcal H}^p_{\Phi,\Omega}(\vec f)\chi_k\|_{L^q_\omega(\mathbb Q^n_p)}^{\ell}\Big)^{1/\ell}\nonumber
\\
&\,\,\,\,\lesssim\|\Omega\|_{L^{q'}(S_0)}\sum\limits_{\gamma\in\mathbb Z}\dfrac{\Phi(p^{\gamma})}{p^{\gamma(1-\frac{n+\alpha}{q})}}\mathop{\rm sup}\limits_{k_0\in\mathbb Z}p^{k_0\lambda^*}\Big(\sum\limits_{k=-\infty}^{k_0}p^{k\beta\ell}\prod\limits_{i=1}^{m}\|f_i\chi_{-\gamma+k}\|^\ell_{L^{q_i}_{\omega_i}(\mathbb Q_p^n)}\Big)^{1/\ell}.
\nonumber
\end{align}
Using the H\"{o}lder inequality and  the relation (\ref{lambdai*}), we get
\begin{align}\label{MorreyHerz3.5}
&p^{-k_0\lambda^*}\Big(\sum\limits_{k=-\infty}^{k_0}p^{k\beta\ell}\prod\limits_{i=1}^{m}\|f_i\chi_{-\gamma+k}\|^\ell_{L^{q_i}_{\omega_i}(\mathbb Q_p^n)}\Big)^{1/\ell}\leq p^{-k_0\lambda^*}\prod\limits_{i=1}^m\Big(\sum\limits_{k=-\infty}^{k_0}p^{k\beta_i\ell_i}\|f_i\chi_{-\gamma+k}\|^{\ell_i}_{L^{q_i}_{\omega_i}(\mathbb Q_p^n)}\Big)^{1/\ell_i}\nonumber
\\
&= p^{\gamma\beta}\prod\limits_{i=1}^mp^{-\gamma\lambda_i}p^{-(k_0-\gamma)\lambda_i}\Big(\sum\limits_{k=-\infty}^{k_0-\gamma}p^{k\beta_i\ell_i}\|f_i\chi_{k}\|^{\ell_i}_{L^{q_i}_{\omega_i}(\mathbb Q_p^n)}\Big)^{1/\ell_i}= p^{\gamma(\beta-\lambda^*)}\prod\limits_{i=1}^m \|f_i\|_{{MK}^{\beta_i,\lambda_i}_{\ell_i, q_i,\omega_i}(\mathbb Q^n_p)}.
\end{align}
Therefore, we obtain
$$
\big\|{\mathcal H}^p_{\Phi,\Omega}(\vec f)\big\|_{{MK}^{\beta,\lambda^*}_{\ell, q,\omega}(\mathbb Q^n_p)}\lesssim \mathcal C_5.\|\Omega\|_{L^{q'}(S_0)}\prod\limits_{i=1}^m \|f_i\|_{{MK}^{\beta_i,\lambda_i}_{\ell_i, q_i,\omega_i}(\mathbb Q^n_p)},
$$
which implies that ${\mathcal H}^p_{\Phi,\Omega}$ is  a bounded operator from ${MK}^{\beta_1,\lambda_1}_{\ell_1, q_1,\omega_1}(\mathbb Q^n_p)\times\cdots \times {MK}^{{\beta_m, \lambda_m}}_{\ell_m, q_m,\omega_m}(\mathbb Q^n_p)$ to ${MK}^{\beta,\lambda^*}_{\ell,q, \omega}(\mathbb Q^n_p)$ if $\mathcal C_5<\infty.$
\\

Conversely, by a similar argument,  we also choose
$$
f_i(x)=|x|_p^{-\beta_i-\frac{n+\alpha_i}{q_i}+\lambda_i}, \textit{ \rm for all}\, i=1,...,m.
$$
Thus, $\|f_i\chi_k\|^{q_i}_{L^{q_i}_{\omega_i}(\mathbb Q^n_p)}=\int_{S_k}|x|_p^{-\beta_iq_i-n+\lambda_iq_i}dx\simeq p^{kq_i(\lambda_i-\beta_i)}$, which implies that
\begin{align}
0<\|f_i\|_{{MK}^{\beta_i,\lambda_i}_{\ell_i,q_i,\omega_i}(\mathbb Q^n_p)}&=\mathop{\rm sup}\limits_{k_0\in\mathbb Z} p^{-k_0\lambda_i}\Big(\sum\limits_{k=-\infty}^{k_0}p^{k\beta_i\ell_i}\|f_i\chi_k\|_{L^{q_i}_{\omega_i}(\mathbb Q^n_p)}^{\ell_i}\Big)^{1/\ell_i}
\\
&\simeq\mathop{\rm sup}\limits_{k_0\in\mathbb Z} p^{-k_0\lambda_i}\Big(\sum\limits_{k=-\infty}^{k_0}p^{k\lambda_i\ell_i}\Big)^{1/\ell_i}=\dfrac{1}{(1-p^{-\lambda_i\ell_i})^{1/\ell_i}}<\infty.
\nonumber
\end{align}
In addition, it is easy to show that
$$
\mathcal H^p_{\Phi,\Omega}(\vec f)(x)=\Big|\int_{S_0}\Omega(y)dy\Big|\mathcal C_5.|x|_p^{-\beta-\frac{n+\alpha}{q}+\lambda^*}.
$$
Hence, we estimate
\begin{align}
\|\mathcal H^p_{\Phi,\Omega}(\vec f)\|_{{MK}^{\beta,\lambda^*}_{\ell,q,\omega}(\mathbb Q^n_p)}&= \mathcal C_5.\Big|\int_{S_0}\Omega(y)dy\Big|\||x|_p^{-\beta-\frac{n+\alpha}{q}+\lambda^*}\|_{{MK}^{\beta,\lambda^*}_{\ell,q,\omega}(\mathbb Q^n_p)}\nonumber
\\
&\gtrsim \mathcal C_5.\Big|\int_{S_0}\Omega(y)dy\Big|\prod\limits_{i=1}^{m} \|f_i\|_{{MK}^{\beta_i,\lambda_i}_{\ell_i,q_i,\omega_i}(\mathbb Q^n_p)}\frac{\prod\limits_{i=1}^{m}(1-p^{-\lambda_i\ell_i})^{1/\ell_i}}{(1-p^{-\lambda\ell})^{1/\ell}}.\nonumber
\end{align}
Since ${\mathcal H}^p_{\Phi,\Omega}$ is  a bounded operator from ${MK}^{\beta_1,\lambda_1}_{\ell_1, q_1,\omega_1}(\mathbb Q^n_p)\times\cdots \times {MK}^{{\beta_m, \lambda_m}}_{\ell_m, q_m,\omega_m}(\mathbb Q^n_p)$ to ${MK}^{\beta,\lambda^*}_{\ell,q, \omega}(\mathbb Q^n_p)$, it follows that
\[
\mathcal C_5\lesssim {\big\|{\mathcal H}^p_{\Phi,\Omega}\big\|_{{MK}^{\beta_1,\lambda_1}_{\ell_1, q_1,\omega_1}(\mathbb Q^n_p)\times\cdots \times {MK}^{{\beta_m, \lambda_m}}_{\ell_m, q_m,\omega_m}(\mathbb Q^n_p)\to {MK}^{\beta,\lambda^*}_{\ell,q, \omega}(\mathbb Q^n_p)}}<\infty,
\]
which finishes the proof.
\end{proof}
\section{The main results about the boundness of ${\mathcal{H}}^p_{\Phi,\Omega,\vec b}$}\label{section4}
Before stating our next results, we introduce some notations which will be used throughout this section. Let $q$ and $q_i\in [1,\infty)$, $\alpha,\beta,\ell$, $\alpha_i, \beta_i, \ell_i, r_i$ are real numbers such that $\alpha_i \in (-n,\infty)$, $\ell,\ell_i\in [1,\infty)$, $r_i\in [1,\infty)$, $i=1,2,...,m$ and denote
$$  \left(\frac{1}{{{q_1}}} + \frac{1}{{{q_2}}} + \cdots + \frac{1}{{{q_m}}}\right) + \left(\frac{1}{{{r_1}}} + \frac{1}{{{r_2}}} + \cdots + \frac{1}{{{r_m}}}\right) = \frac{1}{q}, $$
  $$  \left( \frac{{\alpha_1}}{{{q_{1}}}} + \frac{\alpha _2}{{{q_{2}}}} + \cdots + \frac{\alpha_m}{{{q_{m}}}} \right)+\left( \frac{{\alpha_1}}{{{r_{1}}}} + \frac{\alpha _2}{{{r_{2}}}} + \cdots + \frac{\alpha_m}{{{r_{m}}}}\right)= \frac{\alpha}{q},$$
$$ \beta_1+\cdots+\beta_m = \beta, $$
$$ \dfrac{1}{\ell_1}+\cdots\dfrac{1}{\ell_m}=\dfrac{1}{\ell}.$$
\begin{theorem}\label{theoremMorreyCMO}
Let $\lambda_i\in \big(\frac{-1}{q_{i}},0\big)$, $\omega_i(x)=|x|_p^{\alpha_i}$, $b_i\in {\mathop {CMO}\limits^{.}}^{r_i}_{\omega_i}(\mathbb Q^n_p)$ for all $i=1,...,m$, $\omega(x)=|x|_p^{\alpha}$, $\Omega\in L^{q'}(S_0)$, and the hypothesis (\ref{lambdaMorrey}) in Theorem \ref{TheoremMorrey} holds.
\\
(i) If
$$
\mathcal C_6=\sum\limits_{\gamma\in\mathbb Z}\dfrac{\Phi(p^{\gamma})}{p^{\gamma(1+(n+\alpha)\lambda)}}\prod\limits_{i=1}^{m}\Big(2+p^{|\gamma|(\alpha_i+n)}\Big)<\infty,
$$
then we have ${\mathcal{H}}^p_{\Phi,\Omega,\vec b}$ is bounded from ${\mathop B\limits^.}^{q_1,\lambda _1}_{\omega_1}(\mathbb Q^n_p)\times\cdots \times {\mathop B\limits^.}^{{q_m},{\lambda _m}}_{\omega_m}(\mathbb Q^n_p)$ to ${\mathop B\limits^.}^{q,\lambda}_{\omega}(\mathbb Q^n_p)$.
\\
(ii) If ${\mathcal{H}}^p_{\Phi,\Omega,\vec b}$ is bounded from ${\mathop B\limits^.}^{q_1,\lambda _1}_{\omega_1}(\mathbb Q^n_p)\times\cdots \times {\mathop B\limits^.}^{{q_m},{\lambda _m}}_{\omega_m}(\mathbb Q^n_p)$ to ${\mathop B\limits^.}^{q,\lambda}_{\omega}(\mathbb Q^n_p)$ and 
$$\mathcal C_6^*=\Big|\sum\limits_{\gamma\in\mathbb Z}\dfrac{\Phi(p^{\gamma}){\gamma}^m}{p^{\gamma(1+(n+\alpha)\lambda)}}\Big|,
$$
then $\mathcal C^*_6$ is finite. Furthermore,
$$\|{\mathcal{H}}^p_{\Phi,\Omega,\vec b}\|_{{\mathop B\limits^.}^{q_1,\lambda _1}_{\omega_1}(\mathbb Q^n_p)\times\cdots \times {\mathop B\limits^.}^{{q_m},{\lambda _m}}_{\omega_m}(\mathbb Q^n_p)\to{\mathop B\limits^.}^{q,\lambda}_{\omega}(\mathbb Q^n_p)}\gtrsim \mathcal C_6^*.$$
\end{theorem}
\begin{proof}
By applying the Minkowski inequality, it implies that
\begin{align}
&\|{\mathcal{H}}^p_{\Phi,\Omega,\vec b}(\vec f)\|_{L^q_\omega(B_R)}\nonumber
\\
&\,\,\,\,\,\leq \sum\limits_{\gamma\in\mathbb Z}\dfrac{\Phi(p^{\gamma})}{p^{\gamma}}\int_{S_0}|\Omega(y)|\Big(\int_{B_R}\prod\limits_{i=1}^m|b_i(x)-b_i(p^{\gamma}|x|_p^{-1}y)|^q.|f_i(p^{\gamma}|x|_p^{-1}y)|^q\omega dx\Big)^{1/q}dy.
\nonumber
\end{align}
According to the H\"{o}lder inequality, we get
\begin{align}
&\Big(\int_{B_R}\prod\limits_{i=1}^m|b_i(x)-b_i(p^{\gamma}|x|_p^{-1}y)|^q.|f_i(p^{\gamma}|x|_p^{-1}y)|^q\omega dx\Big)^{1/q}\nonumber
\\
&\,\,\,\,\,\leq \prod\limits_{i=1}^m\|b_i(\cdot)-b_i(p^{\gamma}|\cdot|_p^{-1}y)\|_{L^{r_i}_{\omega_i}(B_R)}.\prod\limits_{i=1}^m\|f_i(p^{\gamma}|\cdot|_p^{-1}y)\|_{L^{q_i}_{\omega_i}(B_R)}.
\nonumber
\end{align}
Therefore, by using the H\"{o}lder inequality again, $\|{\mathcal{H}}^p_{\Phi,\Omega,\vec b}(\vec f)\|_{L^q_\omega(B_R)}$ is controlled by
\begin{align}\label{HfMorCMO}
&\|\Omega\|_{L^{q'}(S_0)}\sum\limits_{\gamma\in\mathbb Z}\dfrac{\Phi(p^{\gamma})}{p^{\gamma}}\prod\limits_{i=1}^m\Big(\int_{S_0}\|b_i(\cdot)-b_i(p^{\gamma}|\cdot|_p^{-1}y)\|^{r_i}_{L^{r_i}_{\omega_i}(B_R)}dy\Big)^{1/r_i}\times
\nonumber
\\
&\,\,\times\prod\limits_{i=1}^m\Big(\int_{S_0}\|f_i(p^{\gamma}|\cdot|_p^{-1}y)\|^{q_i}_{L^{q_i}_{\omega_i}(B_R)}dy\Big)^{1/q_i}.
\end{align}
Next, we can show that
\begin{align}\label{esCMO}
\Big(\int_{S_0}\|b_i(\cdot)-b_i(p^{\gamma}|\cdot|_p^{-1}y)\|^{r_i}_{L^{r_i}_{\omega_i}(B_R)}dy\Big)^{1/r_i}\lesssim \Big(2+p^{|\gamma|(\alpha_i+n)}\Big)p^{\frac{R(\alpha_i+n)}{r_i}}\|b_i\|_{ {\mathop {CMO}\limits^{.}}^{r_i}_{\omega_i}(\mathbb Q^n_p)}.
\end{align}
Indeed, we estimate
\begin{align}
&\Big(\int_{S_0}\|b_i(\cdot)-b_i(p^{\gamma}|\cdot|_p^{-1}y)\|^{r_i}_{L^{r_i}_{\omega_i}(B_R)}dy\Big)^{1/r_i}\leq\Big(\int_{S_0}\|b_i(\cdot)-b_{i,\omega_i,B_{R}}\|^{r_i}_{L^{r_i}_{\omega_i}(B_R)}dy\Big)^{1/r_i}+\nonumber
\\
&+ \Big(\int_{S_0}\|b_i(p^{\gamma}|\cdot|_p^{-1}y)-b_{i,\omega_i,B_{R-\gamma}}\|^{r_i}_{L^{r_i}_{\omega_i}(B_R)}dy\Big)^{1/r_i}+ \Big(\int_{S_0}\|b_{i,\omega_i,B_{R-\gamma}}-b_{i,\omega_i,B_R}\|^{r_i}_{L^{r_i}_{\omega_i}(B_R)}dy\Big)^{1/r_i}\nonumber
\\
&:= I_1+I_2+I_3.
\end{align}
In view of the definition of the space ${\mathop{CMO}\limits^{.}}_{\omega_i}^{r_i}(\mathbb Q^n_p)$, we deduce
\begin{align}\label{esI1}
I_1&=\Big(\omega_i(B_R)\int_{S_0}\frac{1}{\omega_i(B_R)}\|b_i(\cdot)-b_{i,\omega_i,B_{R}}\|^{r_i}_{L^{r_i}_{\omega_i}(B_R)}dy\Big)^{1/r_i}\nonumber
\\
&\leq \Big(\omega_i(B_R)\int_{S_0}\|b_i\|^{r_i}_{{\mathop{CMO}\limits^{.}}_{\omega_i}^{r_i}(\mathbb Q^n_p)}dy\Big)^{1/r_i}\lesssim p^{\frac{R(\alpha_i+n)}{r_i}}\|b_i\|_{{\mathop{CMO}\limits^{.}}_{\omega_i}^{r_i}(\mathbb Q^n_p)}.
\end{align}
Next, we observe that
\begin{align}
I_2^{r_i}&=\int_{S_0}\int_{B_R}|b_i(p^{\gamma}|x|_p^{-1}y)-b_{i,\omega_i,B_{R-\gamma}}|^{r_i}\omega(x)dxdy\nonumber
\\
&=\int_{S_0}\sum\limits_{\theta\leq R}\Big(\int_{S_\theta}|b_i(p^{\gamma}p^{-\theta}y)-b_{i,\omega_i,B_{R-\gamma}}|^{r_i}p^{\theta\alpha_i}dx\Big)dy
\nonumber
\\
&\lesssim p^{\gamma(n+\alpha_i)}\sum\limits_{\theta\leq R}\int_{S_0}|b_i(p^{-(-\gamma+\theta)}y)-b_{i,\omega_i,B_{R-\gamma}}|^{r_i}p^{(-\gamma+\theta)\alpha_i}p^{n(-\gamma+\theta)}dy,
\nonumber
\end{align}
which implies that
\begin{align}\label{esI2before}
I_2 &\lesssim p^{\frac{\gamma(n+\alpha_i)}{r_i}}\Big(\sum\limits_{\zeta\leq R-\gamma}\int_{S_0}|b_i(p^{-\zeta}.y)-b_{i,\omega_i,B_{R-\gamma}}|^{r_i}p^{\zeta\alpha_i}p^{n\zeta}dy\Big)^{1/r_i}\nonumber
\\
&= p^{\frac{\gamma(n+\alpha_i)}{r_i}}\Big(\int_{B_{R-\gamma}}|b_i(x)-b_{i,\omega_i,B_{R-\gamma}}|^{r_i}\omega(x)dx\Big)^{1/r_i}.
\end{align}
Thus, 
\begin{align}\label{esI2}
I_2\lesssim p^{\frac{\gamma(n+\alpha_i)}{r_i}}\omega^{\frac{1}{r_i}}(B_{R-\gamma})\|b_i\|_{{\mathop{CMO}\limits^{.}}_{\omega_i}^{r_i}(\mathbb Q^n_p)}\nonumber\lesssim p^{\frac{R(\alpha_i+n)}{r_i}}\|b_i\|_{{\mathop{CMO}\limits^{.}}_{\omega_i}^{r_i}(\mathbb Q^n_p)}.
\\
\end{align}
Now, we see that
\begin{align}\label{esI3}
I_3&=\Big(\int_{S_0}\int_{B_R}|b_{i,\omega_i,B_{R-\gamma}}-b_{i,\omega_i,B_R}|^{r_i}\omega_i(x)dxdy\Big)^{1/r_i}\lesssim \omega_i^{{1/r_i}}(B_R).|b_{i,\omega_i,B_{R-\gamma}}-b_{i,\omega_i,B_R}|\nonumber
\\
&\lesssim p^{\frac{R(\alpha_i+n)}{r_i}}.|b_{i,\omega_i,B_{R-\gamma}}-b_{i,\omega_i,B_R}|.
\end{align}
In the case $\gamma<0$, by using the H\"{o}lder inequality, it follows that
\begin{align}\label{esI3b1}
|b_{i,\omega_i,B_{R-\gamma}}-b_{i,\omega_i,B_R}|&\leq \dfrac{1}{\omega_i(B_R)}\int_{B_{R}}|b_i(x)-b_{i,\omega_i,B_{R-\gamma}}|\omega_i(x)dx\nonumber
\\
&\leq \dfrac{\omega_i^{\frac{1}{r'_i}}(B_{R-\gamma})}{\omega_i(B_R)}\Big(\int_{B_{R-\gamma}}|b_i(x)-b_{i,\omega_i,B_{R-\gamma}}|^{r_i}\omega_i(x)dx\Big)^{1/r_i}
\nonumber
\\
&\leq \dfrac{\omega_i(B_{R-\gamma})}{\omega_i(B_R)}\|b_i\|_{{\mathop{CMO}\limits^{.}}_{\omega_i}^{r_i}(\mathbb Q^n_p)}\lesssim p^{-\gamma(\alpha_i+n)}\|b_i\|_{{\mathop{CMO}\limits^{.}}_{\omega_i}^{r_i}(\mathbb Q^n_p)}.
\end{align}
Otherwise, by estimating as above, we also have
\begin{align}\label{esI3b2}
|b_{i,\omega_i,B_{R-\gamma}}-b_{i,\omega_i,B_R}|\lesssim p^{\gamma(\alpha_i+n)}\|b_i\|_{{\mathop{CMO}\limits^{.}}_{\omega_i}^{r_i}(\mathbb Q^n_p)}.
\end{align}
By (\ref{esI3}), (\ref{esI3b1}) and (\ref{esI3b2}), we obtain that
\begin{align}
I_3\lesssim p^{\frac{R(\alpha_i+n)}{r_i}}.p^{|\gamma|(\alpha_i+n)}\|b_i\|_{{\mathop{CMO}\limits^{.}}_{\omega_i}^{r_i}(\mathbb Q^n_p)}.
\nonumber
\end{align}
From this, by (\ref{esI1}) and (\ref{esI2}), we finish the proof of the inequality (\ref{esCMO}). In view of the relations (\ref{S0Br}),  (\ref{HfMorCMO}) and (\ref{esCMO}), we imply
\begin{align}\label{HfMorCMO1}
\|{\mathcal{H}}^p_{\Phi,\Omega,\vec b}(\vec f)\|_{L^q_\omega(B_R)}&\leq \|\Omega\|_{L^{q'}(S_0)}\sum\limits_{\gamma\in\mathbb Z}\dfrac{\Phi(p^{\gamma})}{p^{\gamma}}\prod\limits_{i=1}^m \Big(2+p^{{|\gamma|(\alpha_i+n)}}\Big)p^{\frac{R(\alpha_i+n)}{r_i}}\|b_i\|_{ {\mathop {CMO}\limits^{.}}^{r_i}_{\omega_i}(\mathbb Q^n_p)}
\times\nonumber
\\
&\,\,\times\prod\limits_{i=1}^m p^{\frac{\gamma(n+\alpha_i)}{q_i}}\|f_i\|_{L^{q_i}_{\omega_i}(B_{R-\gamma})}.
\end{align}
By having the hypothesis (\ref{lambdaMorrey}), it deduces
$\frac{1}{\omega^{\frac{1}{q}+\lambda}(B_{R})}\simeq \prod\limits_{i=1}^m\frac{p^{\frac{-R(n+\alpha_i)}{r_i}-\gamma(n+\alpha_i)(\frac{1}{q_i}+\lambda_i)}}{\omega_i^{\frac{1}{q_i}+\lambda_i}(B_{R-\gamma})}.
$
Thus, by (\ref{HfMorCMO1}) and the definition of the $\lambda$-central Morrey $p$-adic space, we have
\begin{align}
\|{\mathcal{H}}^p_{\Phi,\Omega,\vec b}(\vec f)\|_{{\mathop B\limits^.}^{q,\lambda}_{\omega}(\mathbb Q^n_p)}&\lesssim  \mathcal C_6.\|\Omega\|_{L^{q'}(S_0)}.\prod\limits_{i=1}^m|b_i\|_{ {\mathop {CMO}\limits^{.}}^{r_i}_{\omega_i}(\mathbb Q^n_p)}.\prod\limits_{i=1}^m \|f_i\|_{{\mathop B\limits^.}^{q_i,\lambda_i}_{\omega_i}(\mathbb Q^n_p)},\nonumber
\end{align}
which completes the proof for the first part of theorem.
\\

Next, we will prove the second part of theorem. Let us take that ${\mathcal{H}}^p_{\Phi,\Omega,\vec b}$ is bounded from ${\mathop B\limits^.}^{q_1,\lambda _1}_{\omega_1}(\mathbb Q^n_p)\times\cdots \times {\mathop B\limits^.}^{{q_m},{\lambda _m}}_{\omega_m}(\mathbb Q^n_p)$ to ${\mathop B\limits^.}^{q,\lambda}_{\omega}(\mathbb Q^n_p)$. We choose
$$
b_i(x)={\rm log}_{p}|x|_p\,\,\textit{\rm and}\,\,f_i(x)=|x|_p^{(n+\alpha_i)\lambda_i},\,\textit{\rm for all}\, i=1,...,m.
$$
Note that, by Lemma 6.1 in \cite{Hung} (also see Lemma 2.1 in \cite{Rim}), we imply $b_i\in {\rm BMO}_\omega(\mathbb Q^n_p)\subset {\rm CMO}^{r_i}_\omega(\mathbb Q^n_p)$. 
From (\ref{estimatefi}), we have $f_i\in{\mathop B\limits^.}^{q_i,\lambda _i}_{\omega_i}(\mathbb Q^n_p)$ and $|x|_p^{(n+\alpha)\lambda}\in{\mathop B\limits^.}^{q,\lambda}_{\omega}(\mathbb Q^n_p)$.
\\
Thus, we get
\begin{align}
{\mathcal H}^p_{\Phi,\Omega,\vec b}(\vec f)(x)&=\sum_{\gamma\in\mathbb Z}\int_{S_0} \dfrac{\Phi(p^{\gamma})}{p^{\gamma}}\Omega(y)\prod\limits_{i=1}^{m}\Big({\rm log}_p|x|_p- {\rm log}_p|p^{\gamma}|x|_p^{-1}y|_p\Big)\prod\limits_{i=1}^{m}|p^{\gamma}|x|_p^{-1}y|_p^{(n+\alpha_i)\lambda_i}dy
\nonumber
\\
&=\Big(\sum_{\gamma\in\mathbb Z} \dfrac{\Phi(p^{\gamma})\gamma^m}{p^{\gamma}}\Big)\Big(\int_{S_0}\Omega(y)dy\Big)|x|_p^{(n+\alpha)\lambda}.\nonumber
\end{align}
Therefore, we deduce
\begin{align}
\|{\mathcal H}^p_{\Phi,\Omega,\vec b}(\vec f)\|_{{\mathop B\limits^.}^{q,\lambda}_{\omega}(\mathbb Q^n_p)}&=\mathcal C_6^*.\Big|\int_{S_0}\Omega(y)dy\Big|.\left\||x|_p^{(n+\alpha)\lambda}\right\|_{{\mathop B\limits^.}^{q,\lambda}_{\omega}(\mathbb Q^n_p)}\nonumber
\\
&\gtrsim \mathcal C_6^*\prod\limits_{i=1}^m|b_i\|_{ {\mathop {CMO}\limits^{.}}^{r_i}_{\omega_i}(\mathbb Q^n_p)}.\prod\limits_{i=1}^m \|f_i\|_{{\mathop B\limits^.}^{q_i,\lambda_i}_{\omega_i}(\mathbb Q^n_p)}.\nonumber
\end{align}
This leads that $\mathcal C_4^*$ is finite. The proof of theorem is ended.
\end{proof}
\begin{theorem}\label{theoremMorreyCMO2}
Let $1\leq q, r_1^*, ..., r_m^*, q_1^*, ..., q_m^*, \zeta<\infty$, $\omega\in A_{\zeta}$ with the finite critical index $r_\omega$ for the reverse H\"{o}lder condition, $\delta\in (1,r_\omega)$, $\Omega\in L^{q'}(S_0)$, $\lambda_i\in \big(\frac{-1}{q^*_{i}},0\big)$ and $b_i\in {\mathop {CMO}\limits^{.}}^{r_i^*}_{\omega}(\mathbb Q^n_p)$ for all $i=1,...,m$. Assume that the hypothesis (\ref{lambdai*}) in Theorem \ref{TheoremMorrey1} holds and the two following conditions are true:
\begin{align}\label{r*q*}
 \dfrac{1}{q}>\Big(\dfrac{1}{r_1^*}+ ...+ \dfrac{1}{r_m^*}+ \dfrac{1}{q_1^*}+ ...+ \dfrac{1}{q_m^*}\Big)\zeta\dfrac{r_\omega}{r_\omega-1},
\end{align}
\begin{align}\label{C7}
\mathcal C_7=\sum\limits_{\gamma\geq 0}\dfrac{\Phi(p^{\gamma})}{p^{\gamma(1+n\zeta\lambda^*)}}(2+p^{\gamma\zeta n})^m+\sum\limits_{\gamma<0}\dfrac{\Phi(p^{\gamma})}{p^{\gamma(1+n\lambda^*(\delta-1)/\delta))}}(2+p^{-\gamma\zeta n})^m<\infty.
\end{align}
Then, ${\mathcal{H}}^p_{\Phi,\Omega,\vec b}$ is bounded from ${\mathop B\limits^.}^{q_1^*,\lambda _1}_{\omega}(\mathbb Q^n_p)\times\cdots \times {\mathop B\limits^.}^{{q_m^*},{\lambda _m}}_{\omega}(\mathbb Q^n_p)$ to ${\mathop B\limits^.}^{q,\lambda^*}_{\omega}(\mathbb Q^n_p)$.
\end{theorem}
\begin{proof}
From the inequality (\ref{r*q*}), there exist $r_1, ..., r_m, q_1, ..., q_m$ such that
$$
\frac{1}{r_i}>\frac{\zeta}{r_i^*}\frac{r_\omega}{r_\omega-1},$$
$$
\frac{1}{q_i}>\frac{\zeta}{q_i^*}\frac{r_\omega}{r_\omega-1},$$
and $$ \sum\limits_{i=1}^m\frac{1}{r_i}+ \frac{1}{q_i}=\frac{1}{q}.
$$
Because of  $\sum\limits_{i=1}^m\dfrac{1}{r_i}+ \dfrac{1}{q_i}=\dfrac{1}{q}$ and making a similar argument as (\ref{HfMorCMO}), we also have
\begin{align}\label{HfMorCMO1}
\|{\mathcal{H}}^p_{\Phi,\Omega,\vec b}(\vec f)\|_{L^q_\omega(B_R)}&\lesssim\|\Omega\|_{L^{q'}(S_0)}\sum\limits_{\gamma\in\mathbb Z}\dfrac{\Phi(p^{\gamma})}{p^{\gamma}}\prod\limits_{i=1}^m\Big(\int_{S_0}\|b_i(\cdot)-b_i(p^{\gamma}|\cdot|_p^{-1}y)\|^{r_i}_{L^{r_i}_{\omega}(B_R)}dy\Big)^{1/r_i}\times
\nonumber
\\
&\,\,\times\prod\limits_{i=1}^m\Big(\int_{S_0}\|f_i(p^{\gamma}|\cdot|_p^{-1}y)\|^{q_i}_{L^{q_i}_{\omega}(B_R)}dy\Big)^{1/q_i}.
\end{align}
To prove this theorem, we need to show the following result,
\begin{align}\label{esCMO1}
\Big(\int_{S_0}\|b_i(\cdot)-b_i(p^{\gamma}|\cdot|_p^{-1}y)\|^{r_i}_{L^{r_i}_{\omega}(B_R)}dy\Big)^{1/r_i}\lesssim \Big(2+p^{|\gamma|n\zeta}\Big)\omega(B_R)^{\frac{1}{r_i}}\|b_i\|_{ {\mathop {CMO}\limits^{.}}^{r_i^*}_{\omega}(\mathbb Q^n_p)}.
\end{align}
Actually, we compose
\begin{align}\label{composeCMO1}
&\Big(\int_{S_0}\|b_i(\cdot)-b_i(p^{\gamma}|\cdot|_p^{-1}y)\|^{r_i}_{L^{r_i}_{\omega}(B_R)}dy\Big)^{1/r_i}\lesssim \Big(\int_{S_0}\|b_i(\cdot)-b_{i,\omega,B_{R}}\|^{r_i}_{L^{r_i}_{\omega}(B_R)}dy\Big)^{1/r_i}\nonumber
\\
&+\Big(\int_{S_0}\|b_i(p^{\gamma}|\cdot|_p^{-1}y)-b_{i,\omega,B_{R-\gamma}}\|^{r_i}_{L^{r_i}_{\omega}(B_R)}dy\Big)^{1/r_i}+ \Big(\int_{S_0}\|b_{i,\omega,B_{R-\gamma}}-b_{i,\omega,B_R}\|^{r_i}_{L^{r_i}_{\omega}(B_R)}dy\Big)^{1/r_i}
\nonumber
\\
&:= J_{1,i}+J_{2,i}+J_{3,i}.
\end{align}
By estimating as (\ref{esI1}) above, we deduce
\begin{align}\label{esJ1end}
J_{1,i}\lesssim \omega(B_R)^{\frac{1}{r_i}}\|b_i\|_{{\mathop{CMO}\limits^{.}}_{\omega}^{r_i}(\mathbb Q^n_p)}.
\end{align}
Since $\frac{1}{r_i}>\frac{\zeta}{r_i^*}\frac{r_\omega}{r_\omega-1}$, there exists $\beta_{i,0}\in (1, r_\omega)$ rewarding $\frac{r_i^*}{\zeta}=r_i\beta_{i,0}'$. By the H\"{o}lder inequality and the reverse H\"{o}lder condition again, we infer
\begin{align}
J_{2,i}^{r_i}&\leq \int_{S_0}\Big(\int_{B_R}{|b_i(p^{\gamma}|x|_p^{-1}y)-b_{i,\omega,B_{R-\gamma}}|^{\frac{r_i^*}{\zeta}}dx\Big)^{\frac{\zeta r_i}{r_i^*}}\Big(\int_{B_R}\omega(x)^{\beta_{i,0}}dx\Big)^{\frac{1}{\beta_{i,0}}}dy}\nonumber
\\
&\leq |B_R|^{\frac{-1}{\beta_{i,0}'}}\omega(B_R)\int_{S_0}\Big(\int_{B_R}{|b_i(p^{\gamma}|x|_p^{-1}y)-b_{i,\omega,B_{R-\gamma}}|^{\frac{r_i^*}{\zeta}}dx\Big)^{\frac{\zeta r_i}{r_i^*}}dy}\nonumber
\\
&\lesssim |B_R|^{\frac{-\zeta r_i}{r^*_i}}\omega(B_R)\Big(\int_{S_0}\int_{B_R}{|b_i(p^{\gamma}|x|_p^{-1}y)-b_{i,\omega,B_{R-\gamma}}|^{\frac{r_i^*}{\zeta}}dxdy}\Big)^{\frac{\zeta r_i}{r_i^*}}.\nonumber
\end{align}
By evaluating as (\ref{esI2before}) and Proposition \ref{pro2.4DFan}, we have
\begin{align}
&\Big(\int_{S_0}\int_{B_R}{|b_i(p^{\gamma}|x|_p^{-1}y)-b_{i,\omega,B_{R-\gamma}}|^{\frac{r_i^*}{\zeta}}dxdy}\Big)^{\frac{\zeta}{r_i^*}}\lesssim p^{\frac{\gamma n\zeta}{r^*_i}}\Big(\int_{B_{R-\gamma}}|b_i(x)-b_{i,\omega,B_{R-\gamma}}|^{\frac{r_i^*}{\zeta}}dx\Big)^{\frac{\zeta}{r_i^*}}\nonumber
\\
&\lesssim p^{\frac{\gamma n\zeta}{r^*_i}}|B_{R-\gamma}|^{\frac{\zeta}{r^*_i}}\omega(B_{R-\gamma})^{\frac{-1}{r^*_i}}\|b_i(\cdot)-b_{i,\omega,B_{R-\gamma}}\|_{L^{r^*_i}_\omega(B_{R-\gamma})}.\nonumber
\end{align}
This deduces that
\begin{align}\label{esJ2end}
J_{2,i}&\lesssim\omega(B_R)^{\frac{1}{r_i}}p^{\frac{\gamma n\zeta}{r^*_i}}\Big(\dfrac{|B_{R-\gamma}|}{|B_R|}\Big)^{\frac{\zeta}{r^*_i}}\omega(B_{R-\gamma})^{\frac{-1}{r^*_i}}\|b_i(\cdot)-b_{i,\omega,B_{R-\gamma}}\|_{L^{r^*_i}_\omega(B_{R-\gamma})}\nonumber
\\
&\lesssim\omega(B_R)^{\frac{1}{r_i}}\|b_i\|_{{\mathop{CMO}\limits^{.}}_{\omega}^{r_i^*}(\mathbb Q^n_p)}.
\end{align}
By the reasons as (\ref{esI3}), (\ref{esI3b1}) and (\ref{esI3b2}) above, we estimate
\begin{align}
J_{3,i}&\lesssim \omega(B_R)^{^{\frac{1}{r_i}}}|b_{i,\omega,B_{R-\gamma}}-b_{i,\omega,B_R}|\lesssim \omega(B_R)^{^{\frac{1}{r_i}}}\left\{ \begin{array}{l}
\dfrac{\omega(B_{R-\gamma})}{\omega(B_R)}\|b_i\|_{{\mathop{CMO}\limits^{.}}_{\omega}^{r_i}(\mathbb Q^n_p)},\,\textit{\rm if} \,\gamma<0
\\
\\
\dfrac{\omega(B_R)}{\omega(B_{R-\gamma})}\|b_i\|_{{\mathop{CMO}\limits^{.}}_{\omega}^{r_i}(\mathbb Q^n_p)},\,\textit{\rm otherwise}
\end{array} \right..\nonumber
\end{align}
Because of giving $\omega\in A_{\zeta}$ and using Proposition \ref{pro2.3DFan}, we have 
\begin{align}
\left\{ \begin{array}{l}
\dfrac{\omega(B_{R-\gamma})}{\omega(B_R)}\lesssim \Big(
\dfrac{|B_{R-\gamma}|}{|B_R|}\Big)^{\zeta}\lesssim p^{-\gamma n\zeta},\,\textit{\rm if} \,\gamma<0
\\
\\
\dfrac{\omega(B_R)}{\omega(B_{R-\gamma})}\lesssim \Big(
\dfrac{|B_{R}|}{|B_{R-\gamma}|}\Big)^{\zeta}\lesssim p^{\gamma n\zeta},\,\textit{\rm otherwise}
\end{array} \right..\nonumber
\end{align}
Hence, we have
\begin{align}\label{esJ3end}
J_{3,i}\lesssim \omega(B_R)^{^{\frac{1}{r_i}}}p^{|\gamma| n\zeta}\|b_i\|_{{\mathop{CMO}\limits^{.}}_{\omega}^{r_i}(\mathbb Q^n_p)}.
\end{align}
Note that, from the inequality $r_i<r^*_i$, we get $\|b_i\|_{{\mathop{CMO}\limits^{.}}_{\omega}^{r_i}(\mathbb Q^n_p)}\lesssim \|b_i\|_{{\mathop{CMO}\limits^{.}}_{\omega}^{r^*_i}(\mathbb Q^n_p)}.$
Therefore, by (\ref{esJ1end}), (\ref{esJ2end}) and (\ref{esJ3end}), we obtain the proof of the inequality (\ref{esCMO1}). 
\\

On the other hand, because of having $\frac{1}{q_i}>\frac{1}{q_i^*}\zeta\dfrac{r_\omega}{r_\omega-1}$, there exists $\eta_i\in (1, r_\omega)$ such that $\frac{q_i^*}{\zeta}=q_i\eta_i'$. By estimating as (\ref{esJ2end}) above, we infer
\begin{align}
\Big(\int_{S_0}\|f_i(p^{\gamma}|\cdot|_p^{-1}y)\|^{q_i}_{L^{q_i}_{\omega}(B_R)}dy\Big)^{1/q_i}&\lesssim\omega(B_R)^{\frac{1}{q_i}}p^{\frac{\gamma n\zeta}{q^*_i}}\Big(\dfrac{|B_{R-\gamma}|}{|B_R|}\Big)^{\frac{\zeta}{q^*_i}}\omega(B_{R-\gamma})^{\frac{-1}{q^*_i}}\|f_i\|_{L^{q^*_i}_\omega(B_{R-\gamma})}\nonumber
\\
&\lesssim\omega(B_R)^{\frac{1}{q_i}}\omega(B_{R-\gamma})^{\frac{-1}{q^*_i}}\|f_i\|_{L^{q^*_i}_\omega(B_{R-\gamma})}.\nonumber
\end{align}
From this, by (\ref{HfMorCMO1}) and (\ref{esCMO1}), one has
\begin{align}
\|{\mathcal{H}}^p_{\Phi,\Omega,\vec b}(\vec f)\|_{L^q_\omega(B_R)}&\lesssim \|\Omega\|_{L^{q'}(S_0)}\sum\limits_{\gamma\in\mathbb Z}\dfrac{\Phi(p^{\gamma})}{p^{\gamma}}\Big(2+p^{|\gamma|n\zeta}\Big)^m\omega(B_R)^{\frac{1}{q}}\prod\limits_{i=1}^m \|b_i\|_{ {\mathop {CMO}\limits^{.}}^{r^*_i}_{\omega}(\mathbb Q^n_p)}
\times\nonumber
\\
&\,\,\times\prod\limits_{i=1}^m\omega(B_{R-\gamma})^{\frac{-1}{q^*_i}}\|f_i\|_{L^{q^*_i}_{\omega}(B_{R-\gamma})}.\nonumber
\end{align}
As a consequence, by (\ref{lambdai*}), for all $R\in\mathbb Z$, we have
\begin{align}
\dfrac{1}{\omega(B_R)^{1/q+\lambda^*}}\|{\mathcal{H}}^p_{\Phi,\Omega,\vec b}(\vec f)\|_{L^q_\omega(B_R)}&\lesssim \|\Omega\|_{L^{q'}(S_0)}\Big(\sum\limits_{\gamma\in\mathbb Z}\dfrac{\Phi(p^{\gamma})}{p^{\gamma}}\Big(2+p^{|\gamma|n\zeta}\Big)^m\Big(\dfrac{\omega(B_{R-\gamma})}{\omega(B_{R})}\Big)^{\lambda_*}\Big)\times
\nonumber
\\
&\,\,\,\,\,\,\,\,\,\,\,\times\prod\limits_{i=1}^m \|b_i\|_{ {\mathop {CMO}\limits^{.}}^{r^*_i}_{\omega}(\mathbb Q^n_p)}.\prod\limits_{i=1}^m\|f_i\|_{{\mathop B\limits^.}^{q_i^*,\lambda_i}_{\omega}(\mathbb Q^n_p)}.\nonumber
\end{align}
Hence, by (\ref{BRGa>=0}) and (\ref{BRGa<0}), we have
\begin{align}
\|{\mathcal{H}}^p_{\Phi,\Omega,\vec b}(\vec f)\|_{{\mathop B\limits^.}^{q,\lambda^*}_{\omega}(\mathbb Q^n_p)}&\lesssim  \mathcal C_7.\|\Omega\|_{L^{q'}(S_0)}\prod\limits_{i=1}^m \|b_i\|_{ {\mathop {CMO}\limits^{.}}^{r^*_i}_{\omega}(\mathbb Q^n_p)}.\prod\limits_{i=1}^m\|f_i\|_{{\mathop B\limits^.}^{q_i^*,\lambda_i}_{\omega}(\mathbb Q^n_p)},\nonumber
\end{align}
which achieves the desired result.
\end{proof}
\begin{theorem}\label{theoremMHerzCMO}
Let $\omega_i(x)=|x|_p^{\alpha_i}$, $b_i\in {\mathop {CMO}\limits^{.}}^{r_i}_{\omega_i}(\mathbb Q^n_p)$, $\lambda_i\geq 0$ for all $i=1,...,m$, $\omega(x)=|x|_p^{\alpha}$, $\Omega\in L^{q'}(S_0)$. Assume that the hypothesis (\ref{lambdai*}) in Theorem \ref{TheoremMorrey1} is true and the following conditions hold:
\begin{align}\label{beta*}
\beta^*=\beta_1+\cdots+\beta_m-\dfrac{n+\alpha_1}{r_1}-\cdots-\dfrac{n+\alpha_m}{r_m},
\end{align}
\begin{align}
\mathcal C_8=\sum\limits_{\gamma\in\mathbb Z}\dfrac{\Phi(p^{\gamma})}{p^{\gamma(1-\beta^*-\frac{(n+\alpha)}{q}+\lambda^*)}}\prod\limits_{i=1}^{m}\Big(2+p^{|\gamma|(\alpha_i+n)}\Big)<\infty.
\end{align}
Then, we have 
\[
\|{\mathcal{H}}^p_{\Phi,\Omega,\vec b}(\vec f)\|_{{MK}^{\beta^*,\lambda^*}_{\ell, q,\omega}(\mathbb Q^n_p)}\lesssim \mathcal C_8.\|\Omega\|_{L^{q'}(S_0)}.\prod\limits_{i=1}^m\|b_i\|_{ {\mathop {CMO}\limits^{.}}^{r_i}_{\omega_i}(\mathbb Q^n_p)}.\prod\limits_{i=1}^m \|f_i\|_{{MK}^{\beta_i,\lambda_i}_{\ell_i, q_i,\omega_i}(\mathbb Q^n_p)}.
\]
\end{theorem}
\begin{proof}
By estimating as (\ref{HfMorCMO1}) above, it yields that
\begin{align}
\|{\mathcal{H}}^p_{\Phi,\Omega,\vec b}(\vec f)\chi_k\|_{L^q_\omega(\mathbb Q^n_p)}&\lesssim \|\Omega\|_{L^{q'}(S_0)}\sum\limits_{\gamma\in\mathbb Z}\dfrac{\Phi(p^{\gamma})}{p^{\gamma}}\prod\limits_{i=1}^m \Big(2+p^{|\gamma|(\alpha_i+n)}\Big)p^{\frac{k(\alpha_i+n)}{r_i}}\|b_i\|_{ {\mathop {CMO}\limits^{.}}^{r_i}_{\omega_i}(\mathbb Q^n_p)}
\times\nonumber
\\
&\,\,\times\prod\limits_{i=1}^m p^{\frac{\gamma(n+\alpha_i)}{q_i}}\|f_i\chi_{k-\gamma}\|_{L^{q_i}_{\omega_i}(\mathbb Q^n_p)}.\nonumber
\end{align}
From this, by the definition of the Morrey-Herz $p$-adic space and the relation (\ref{beta*}), we obtain that
\begin{align}
\|{\mathcal{H}}^p_{\Phi,\Omega,\vec b}(\vec f)\|_{{MK}^{\beta^*,\lambda^*}_{\ell, q, \omega}(\mathbb Q^n_p)}&\lesssim \|\Omega\|_{L^{q'}(S_0)}\sum\limits_{\gamma\in\mathbb Z}\dfrac{\Phi(p^{\gamma})}{p^{\gamma(1-\sum\limits_{i=1}^m(\frac{n+\alpha_i}{q_i}))}}\prod\limits_{i=1}^m \Big(2+p^{|\gamma|(\alpha_i+n)}\Big)\times
\nonumber
\\
&\times\|b_i\|_{ {\mathop {CMO}\limits^{.}}^{r_i}_{\omega_i}(\mathbb Q^n_p)}\mathop{\rm sup}\limits_{k_0\in\mathbb Z}p^{-k_0\lambda^*}\Big(\sum\limits_{k=-\infty}^{k_0}\prod\limits_{i=1}^m p^{k\beta_i\ell}\|f_i\chi_{k-\gamma}\|^{\ell}_{L^{q_i}_{\omega_i}(\mathbb Q^n_p)}\Big)^{1/\ell}.\nonumber
\end{align}
Thus, by (\ref{MorreyHerz3.5}) above, we get 
$$
\|{\mathcal{H}}^p_{\Phi,\Omega,\vec b}(\vec f)\|_{{MK}^{\beta^*,\lambda^*}_{\ell, q,\omega}(\mathbb Q^n_p)}\lesssim \mathcal C_8.\|\Omega\|_{L^{q'}(S_0)}.\prod\limits_{i=1}^m\|b_i\|_{ {\mathop {CMO}\limits^{.}}^{r_i}_{\omega_i}(\mathbb Q^n_p)}.\prod\limits_{i=1}^m \|f_i\|_{{MK}^{\beta_i,\lambda_i}_{\ell_i, q_i,\omega_i}(\mathbb Q^n_p)},
$$ 
which finishes the proof of Theorem \ref{theoremMHerzCMO}.
\end{proof}
It is well known that the Herz $p$-adic space is a special case of Morrey-Herz $p$-adic space. From this and Theorem \ref{theoremMHerzCMO}, we also obtain the desired result as follows.
\begin{corollary}\label{CoroHerzCMO}
The assumptions of Theorem \ref{theoremMHerzCMO} are true with $\lambda_1=\cdots=\lambda_m =0$ and the following condition holds:
\begin{align}
\mathcal C_9=\sum\limits_{\gamma\in\mathbb Z}\dfrac{\Phi(p^{\gamma})}{p^{\gamma(1-\beta^*-\frac{(n+\alpha)}{q})}}\prod\limits_{i=1}^{m}\Big(2+p^{|\gamma|(\alpha_i+n)}\Big)<\infty.\nonumber
\end{align}
Then, we have ${\mathcal{H}}^p_{\Phi,\Omega,\vec b}$ is bounded from ${K}^{\beta_1,\ell_1}_{q_1,\omega_1}(\mathbb Q^n_p)\times\cdots \times {K}^{{\beta_m, \ell_m}}_{q_m,\omega_m}(\mathbb Q^n_p)$ to ${K}^{\beta^*,\ell}_{q, \omega}(\mathbb Q^n_p)$.
\end{corollary}
\bibliographystyle{amsplain}

\begin{thebibliography}{79}

\bibitem{Albeverio}  S. Albeverio, A.Yu. Khrennikov, V.M. Shelkovich, \textit{Harmonic analysis in the $p$-adic Lizorkin spaces: fractional operators, pseudo-differential equations, $p$-wavelets, Tauberian theorems}, J.  Fourier Anal.  Appl. 12(4) (2006), 393-425.

\bibitem{Avetisov1} A. V. Avetisov, A. H. Bikulov, S. V. Kozyrev, V. A. Osipov, \textit {$p$-adic models of ultrametric diffusion constrained by hierarchical energy landscapes}, J. Phys. A: Math. Gen. 35 (2002), 177-189.

\bibitem{Avetisov2} A. V. Avetisov, A. H. Bikulov, V. A. Osipov, \textit{$p$-adic description of characteristic relaxation in complex systems}, J. Phys. A: Math. Gen. 36 (2003), 4239-4246.

\bibitem{Andersen}K. F. Andersen, B. Muckenhoupt, \textit{Weighted weak type Hardy inequalities with applications to Hilbert transforms and maximal functions}, Studia Math. 72(1)(1982), 9-26.

\bibitem{Beloshapka} O. V. Beloshapka, \textit{ Feynman formulas for the Schr\"{o}dinger equations with the Vladimirov operator}, Russian J. Math. Phys., 17(3), 2010, 267-271.


\bibitem{Coifman} R. R.  Coifman, R. Rochberg, G. Weiss, \textit{Factorization theorems for Hardy spaces in several variables}, Ann. Math. 103(1976), 611-635.

\bibitem{CFL2012} J. Chen, D. Fan and J. Li, \textit{Hausdorff operators on function spaces}, Chinese Annals of Mathematics, Series B. 33(2012), 537-556.

\bibitem{Chuong2018} N. M. Chuong, \textit{Pseudodifferential operators and wavelets over real and p-adic fields}, Springer-Basel, 2018.

\bibitem {Chuong1}N. M. Chuong, Yu. V. Egorov, A. Yu. Khrennikov, Y.  Meyer, D. Mumford,  \textit{ Harmonic, wavelet and $p$-adic analysis}, World Scientific, (2007).

\bibitem {Chuong2} N. M. Chuong, D. V.  Duong, \textit{Wavelet bases in the Lebesgue spaces on the field of p-adic numbers},$p-$Adic numbers, Ultrametric Anal. Appl., Vol. 5(2013), 2, 106-121.

\bibitem{Chuong3} N. M. Chuong,  D. V. Duong, \textit{Weighted Hardy- Littlewood operators and commutators on $p$-adic functional spaces}, $p-$Adic numbers, Ultrametric Anal. Appl., 5(1)(2013), 65-82.


\bibitem {Chuongduong} N. M. Chuong, D. V. Duong, \textit{The $p$-adic weighted Hardy-Ces\`{a}ro operators on weighted Morrey-Herz space}, $p-$Adic numbers, Ultrametric Anal. Appl., 8(3)(2016), 204-216.

\bibitem {Chuong5}  N. M. Chuong, N. V.  Co, \textit{The Cauchy problem for a class of pseudo-differential equations over $p$-adic field,} J. Math. Anal.  Appl.  340(2008), 1, 629-643.

\bibitem{chuonghung}  N. M. Chuong and H. D. Hung, \textit{Maximal functions and weighted norm inequalities on Local Fields}, Appl. Comput. Harmon. Anal. 29 (2010), 272-286.

\bibitem {Dragovich}B. Dragovich, A. Yu. Khrennikov, S. V. Kozyrev, I. V. Volovich, \textit{On $p$-adic mathematical physics}, $p-$Adic numbers, Ultrametric Anal. Appl.,1(1)(2009), 1-17.

\bibitem{Grafakos}L. Grafakos, \textit{Modern Fourier analysis}, Second Edition, Springer, (2008).

\bibitem{Hung} H. D. Hung, \textit{The $p$-adic weighted Hardy-Ces\`{a}ro operator and an application to discrete Hardy inequalities}, J. Math. Anal.  Appl.  409(2014), 868-879.

\bibitem{HCE2012} T. Hyt\"{o}nen, C. P\'{e}rez and E. Rela, \textit{Sharp reverse H\"{o}lder property for $A_{\infty}$ weights on spaces of homogeneous type}, J. Funct. Anal. 263 (2012), 3883-3899.

\bibitem{IMS2015}S. Indratno, D. Maldonado and S. Silwal,\textit{A visual formalism for weights satisfying reverse inequalities}, Expo. Math. 33 (2015), 1-29.

\bibitem{Haran1} S. Haran, \textit{Riesz potentials and explicit sums in arithmetic,}  Invent.  Math. 101(1990), 697-703.

\bibitem{Haran2} S. Haran, \textit{Analytic potential theory over the $p$-adics}, Ann. Inst. Fourier (Grenoble) 43(4) (1993), 905-944.

\bibitem{Khrennikov1} A.Yu.  Khrennikov, \textit{ $p$-adic valued distributions in mathematical physics}, Kluwer Academic Publishers, Dordrecht-Boston-London, (1994). 

\bibitem {Khrennikov2} A.Yu. Khrennikov, V.M. Shelkovich, M. Skopina, \textit{ $p$-Adic refinable functions and MRA-based wavelets,} J.  Approx. Theory 161 (2009), 226-238.

\bibitem{Kozyrev} S. V. Kozyrev,  \textit{Methods and applications of ultrametric and $p$-adic analysis: From wavelet theory to biophysics}, Proc. Steklov Inst. Math., 274(2011), 1-84.

\bibitem{Fu} Z. W. Fu, Z. G. Liu, S. Z. Lu, \textit{Commutators of weighted Hardy operators}, Proc. Amer. Math. Soc. 137(2009), 3319-3328.

\bibitem{FWL2013} Z. W. Fu, Q. Y. Wu, S.Z. Lu, \textit{Sharp estimates of $p$-adic Hardy and Hardy-Littlewood-P\'{o}lya operators}, Acta Math. Sin. 29(2013), 137-150.

\bibitem{LDY2007} S. Lu, Y. Ding and D. Yan, \textit{Singular integrals and related topics}, World Scientific Publishing Company,
Singapore, 2007.

\bibitem{Morrey} C. Morrey, \textit{On the solutions of quasi-linear elliptic partial differential equations}, Trans. Amer. Math. Soc. 43(1938), 126-166.

\bibitem{Muckenhoupt1972} B. Muckenhoupt, \textit{Weighted norm inequalities for the Hardy maximal function}, Trans. Amer. Math. Soc.
165 (1972), 207–226.

\bibitem{RFW2017} J. Ruan, D. Fan and Q. Wu, \textit{Weighted Herz space estimates for Hausdorff operators on the Heisenberg group}, Banach J. Math. Anal. 11 (2017), 513-535.

\bibitem{RFW2017-1} J. Ruan, D. Fan and Q. Wu, \textit{Weighted Morrey estimates for Hausdorff operator and its commutator on the Heisenberg group}, 2017, arXiv:1712.10328.

\bibitem{Rim} K. S. Rim and J. Lee, \textit{Estimates of weighted Hardy-Littlewood averages on the $p$-adic vector space}, J. Math. Anal. Appl. 324(2)(2006), 1470-1477.

\bibitem {Stein}  Elias M. Stein, \textit {Harmonic analysis, real-variable methods, orthogonality, and oscillatory integrals,} Princeton University Press, (1993).

\bibitem{Tang} C. Tang, F. Xue, Y. Zhou, \textit{Commutators of weighted Hardy operators on Herz-type spaces}, Ann. Polon. Math., 101 (2011), 267-273.

\bibitem{Varadarajan} V. S. Varadarajan, \textit {Path integrals for a class of $p$-adic Schr\"{o}dinger equations}, Lett. Math. Phys. 39(1997), 97-106.

\bibitem{Vladimirov}V. S. Vladimirov, \textit {Tables of Integrals of Complex-Valued Functions of $p$-Adic Arguments}, Proc. Steklov Inst. Math. 284(2), 2014, 1-59.

\bibitem{Vladimirov1}V. S. Vladimirov, I.V. Volovich, \textit{$p$-adic quantum mechanics}, Comm. Math. Phys. 123(1989), 659-676.

 \bibitem{Vladimirov2} V. S. Vladimirov, I.V. Volovich, and E.I. Zelenov, \textit{ $p$-adic analysis and mathematical physis}, World Scientific, (1994).

 \bibitem{Volosivets1}  S. S. Volosivets, \textit{Multidimensional Hausdorff operator on $p$-adic field}, $p$-Adic numbers, Ultrametric Anal. Appl. 2 (2010), 252-259.


 \bibitem{Volosivets2} S. S. Volosivets, \textit{Hausdorff operator of special kind in Morrey and Herz $p$-adic spaces}, $p$-Adic numbers, Ultrametric Anal. Appl. 4(2012), 222-230.

\bibitem{Volosivets3} S. S. Volosivets, \textit{Hausdorff operators on p-adic linear spaces and their properties in Hardy, BMO, and H\"{o}lder spaces}, Mathematical Notes, 3(2013), 382-391.

\bibitem{Volosivets4} S. S. Volosivets, \textit{Weak and strong estimates for rough Hausdorff type operator defined on $p$-adic linear space}, $p$-Adic numbers, Ultrametric Anal. Appl. 9(3)(2017), 236-241

\bibitem{WMF2013} Q. Y. Wu, L. Mi and Z. W. Fu, \textit{Boundedness of p-adic Hardy operators and their commutators on p-adic central Morrey and BMO spaces},  J. Funct. Spaces Appl. 2013(2013), Article ID 359193, 10 pages.

\bibitem{Xiao} J. Xiao, \textit{$L^p$ and BMO bounds of weighted Hardy-Littlewood averages}, J. Math. Anal. Appl. 262(2001), 660-666.

\end{thebibliography}

\end{document}